\documentclass[a4paper,10pt]{article} 

\usepackage[utf8]{inputenc} 
\usepackage[T1]{fontenc} 
\usepackage[english]{babel} 
\usepackage{amsmath} 
\usepackage{amsthm}
\usepackage{amssymb}

\usepackage{nicefrac} 
\usepackage{hyperref}
\usepackage{fullpage}

\allowdisplaybreaks

\usepackage{graphicx}
\usepackage{braket}
\usepackage{subcaption}
\usepackage{csquotes}
\usepackage{esint}
\usepackage{color,xcolor}
\usepackage{comment}
\newcommand{\abs}[1]{\left|#1\right|}
\newcommand{\norm}[1]{\left \| #1\right \|}
\newcommand{\R}{\mathbb{R}}
\newcommand{\eps}{\varepsilon}
\newcommand{\ue}{U_\varepsilon}
\newcommand{\ve}{V_\varepsilon}
\newcommand{\we}{W_\varepsilon}
\newcommand{\ze}{Z_\varepsilon}
\newcommand{\phie}{\Phi_\varepsilon}
\newcommand{\psie}{\Psi_\varepsilon}
\newcommand{\tildeue}{\tilde U_\varepsilon}
\newcommand{\tildeve}{\tilde V_\varepsilon}

\numberwithin{equation}{section}

\usepackage{mathtools}
\theoremstyle{plain}
\newtheorem{theorem}{Theorem}[section]

\theoremstyle{plain}
\newtheorem{definition}[theorem]{Definition}

\theoremstyle{plain}
\newtheorem{lemma}[theorem]{Lemma}

\theoremstyle{plain}

\theoremstyle{plain}
\newtheorem{proposition}[theorem]{Proposition}

\theoremstyle{remark}
\newtheorem{remark}[theorem]{Remark}

\def\sideremark#1{\ifvmode\leavevmode\fi\vadjust{\vbox to0pt{\vss% the remark
 \hbox to 0pt{\hskip\hsize\hskip1em%                          will appear only
 \vbox{\hsize2.1cm\tiny\raggedright\pretolerance10000%          on the side
  \noindent #1\hfill}\hss}\vbox to15pt{\vfil}\vss}}}%

\title{Existence of solutions on the critical hyperbola for a pure Lane-Emden system with  Neumann boundary conditions}
\author{Angela Pistoia, Delia Schiera, and Hugo Tavares}
\date{\today}

\begin{document}
\maketitle
%\section{}
\begin{abstract}
We study the following Lane-Emden system 
\[ -\Delta u=|v|^{q-1}v \quad \text{ in } \Omega, \qquad -\Delta v=|u|^{p-1}u \quad \text{ in } \Omega, \qquad u_\nu=v_\nu=0 \quad \text{ on } \partial \Omega, \]
with $\Omega$ a bounded regular domain of $\R^N$, $N \ge 4$, and exponents $p, q$ belonging to the so-called critical hyperbola $1/(p+1)+1/(q+1)=(N-2)/N$.
We show that, under suitable conditions on $p, q$, least-energy (sign-changing) solutions exist, and they are classical. 
In the proof we exploit a dual variational formulation which allows to deal with the strong indefinite character of the problem. We establish a compactness condition which is based on a new Cherrier type inequality. We then prove such condition by using as test functions the solutions to the system in the whole space and performing delicate asymptotic estimates. If $N \ge 5$, $p=1$, the system above reduces to a biharmonic equation, for which we also prove existence of least-energy solutions.  Finally, we prove some partial symmetry and symmetry-breaking results in the case $\Omega$ is a ball or an annulus.
%if $\Omega$ is an annulus, we prove that least energy solutions are not radially symmetric, although there exist radial ones for any $p,q>0$.  
\end{abstract}

\medskip 
\noindent \\
\begin{MScodes}
35J50, 35B33, 35B38, 35J30, 35J47, 35B06.
\end{MScodes}

\section{Introduction}
Consider a bounded domain $\Omega$ of class $C^2$ in $\R^N$, $N \ge 4$, and denote with $\nu$ the outward pointing normal on $\partial \Omega$. Our problem reads as follows 
\begin{equation}\label{system} 
-\Delta u= \abs{v}^{q-1}v  \text{ in } \Omega,\qquad
-\Delta v= \abs{u}^{p-1}u  \text{ in } \Omega,\qquad
u_\nu=v_\nu=0  \text{ on } \partial \Omega, \end{equation}
with
\begin{equation}\label{CH} \frac{1}{p+1} +\frac{1}{q+1}  = \frac{N-2}{N}. \end{equation}
Condition \eqref{CH} is, for these systems, the correct notion of criticality, and we say that $(p, q)$ lies on the \emph{critical hyperbola}. The origin of this concept dates back to the papers \cite{CFM, PvV}, see also the survey \cite{BST} for several considerations and motivations behind it.

We recall that a strong solution to  \eqref{system} is a pair 
\[ (u, v) \in W^{2, \frac{q+1}{q}} (\Omega) \times W^{2, \frac{p+1}{p}} (\Omega) \]
 satisfying \eqref{system} a.e. (the boundary condition being understood in the sense of traces). An important observation is that solutions to \eqref{system} necessarily satisfy the compatibility condition:
\[
\int_\Omega\abs{v}^{q-1}v = \int_\Omega\abs{u}^{p-1}u=0.
\]
Therefore, if $(u,v)\neq(0,0)$ is a strong solution, then necessarily both $u$ and $v$ are nodal (\emph{i.e.} sign-changing) functions.  

Observe also that, under \eqref{CH}, we have the continuous (but not compact) embeddings 
\begin{equation}\label{eq:Sobolevembeddings}
W^{2, \frac{q+1}{q}} (\Omega)\hookrightarrow L^{p+1}(\Omega),\quad \text{ and } \quad  W^{2, \frac{p+1}{p}} (\Omega)\hookrightarrow L^{q+1}(\Omega).
\end{equation}
Therefore, the associated Euler-Lagrange functional
\[ I(u, v):= \int_\Omega \nabla u \cdot \nabla v - \frac{1}{p+1} \int_\Omega |u|^{p+1} - \frac{1}{q+1} \int_\Omega |v|^{q+1} \]
is finite at strong solutions. 

\begin{definition} 
Let $p,q$ satisfy \eqref{CH}. A least energy (nodal) solution for \eqref{system} is a strong solution which achieves the least energy (nodal) level
\[
c_{p,q}:=\inf\{I(u,v):\ (u,v)\in W^{2, \frac{q+1}{q}} (\Omega) \times W^{2, \frac{p+1}{p}} (\Omega) \text{ strong solution of } \eqref{system}\}.
\]
 \end{definition}
The existence of least energy nodal solutions is known in the subcritical case: $\frac{1}{p+1}+\frac{1}{q+1}>\frac{N-2}{N}$ with $pq\neq 1$, see \cite{ST2}. The goal of our work is to extend this to the critical case. In this direction, the main result of our paper is the following:
\begin{theorem}\label{main thm} 
Let $p, q$ satisfy \eqref{CH}, and moreover
\begin{itemize}
\item[(i)] $N \ge 6$ and 
$
p,q > \frac{N+2}{2(N-2)}$, 
or
\item[(ii)] $N=5$ and $p, q> \frac{17}{13}$, or
\item[(iii)] $N=4$ and $p, q > \frac 7 3$. 
\end{itemize} 
 Then there exists a least energy (nodal) classical solution to \eqref{system}.
\end{theorem}

We point out that we prove regularity of all possible solutions (not only of least energy) for every $p, q$ satisfying \eqref{CH}, see the next result. This is not trivial due to the coupling in the system and since either $p$ and $q$ might be smaller than 1. Previously, the regularity  was known in the subcritical superlinear case \cite{ST2}, and in the critical scalar case $p=q=2^*$ \cite{ST}; the arguments used in these situations (respectively a bootstrap and a Brezis-Kato argument) cannot be used in our situation.   %and it is the content of the following 
\begin{proposition}\label{reg2}
Let $(u, v)$ be a strong solution to \eqref{system}
where $p, q$ satisfy \eqref{CH}. 
Then $(u, v) \in C^{2, \zeta}(\overline \Omega) \times C^{2, \eta}(\overline \Omega)$, with: $\zeta < q$ if $0<q <1$, and $\zeta \in (0, 1)$ if $q \ge 1$; $\eta< p$ if $0<p<1$, and $\eta \in (0, 1)$ if $p \ge 1$. 
%In particular, $(u, v)$ is a classical solution.  
\end{proposition}

The additional restrictions in Theorem \ref{main thm} come from the decay of the solutions in the whole space, which depend on the position of $(p,q)$ in the critical hyperbola.

If we take $N \ge 5$, $p=1$, $q = \frac{N+4}{N-4}$, or equivalently, $q=1$, $p = \frac{N+4}{N-4}$, system \eqref{system} reduces to the fourth order equation
\begin{equation}\label{biharmonic} 
\Delta^2 u = |u|^{\frac{8}{N-4}} u \text{ in } \Omega,\qquad
u_\nu=(\Delta u)_\nu= 0  \text{ on } \partial \Omega.
 \end{equation}
If $N > 6$, then 
\[  \frac{N+2}{2(N-2)} < \frac{N+4}{N-4}, \quad \text{ and } \quad \frac{N+2}{2(N-2)} < 1, \]
hence \eqref{biharmonic}
is contained in Theorem \ref{main thm} if $N >6$. 
However, the case $N=5, 6$ and $p=1$ is not included. Nonetheless, we prove the following.

\begin{theorem}\label{main thm 2}
    Let $N\geq 5$. Then there exists a least energy (nodal) classical solution to problem \eqref{biharmonic}. 
\end{theorem}
In particular, this shows the existence of solutions to \eqref{biharmonic} also in the dimensions $N=5, 6$, however the proof of this result actually works for all $N\geq 5$, therefore it also gives a different proof of existence in case $N \ge 7$. 
Although, up to our knowledge, apparently there are no results regarding \eqref{biharmonic} in the literature, we recall that in  \cite{BCN} a related problem with a Paneitz-Branson type  operator, $Lw:=\Delta^2 u-\Delta u+\alpha u$ with $\alpha>0$, and the same boundary conditions was studied. 

By a perturbation argument, we also show the following 
\begin{theorem}\label{main thm1.2} 
Let $N=5,6$. There exists $\eps=\eps(N,\Omega)$ such that, if $p,q$  satisfy \eqref{CH} and either
\[
|p-1|+\abs{q-\frac{N+4}{N-4}}<\eps \quad \text{ or } \quad \abs{p-\frac{N+4}{N-4}}+|q-1|<\eps,
\]
then there exists a least energy (nodal) classical solution to \eqref{system}.
\end{theorem}

\smallbreak

In Figures \ref{fig:56} and \ref{fig:Nge7} we illustrate the assumptions of Theorems \ref{main thm} and \ref{main thm1.2}.  We now make some comments regarding the proof of Theorem \ref{main thm}.

\begin{figure}
 \centering 
\begin{subfigure}{0.48 \textwidth}
    \centering
 \includegraphics[width=\linewidth]{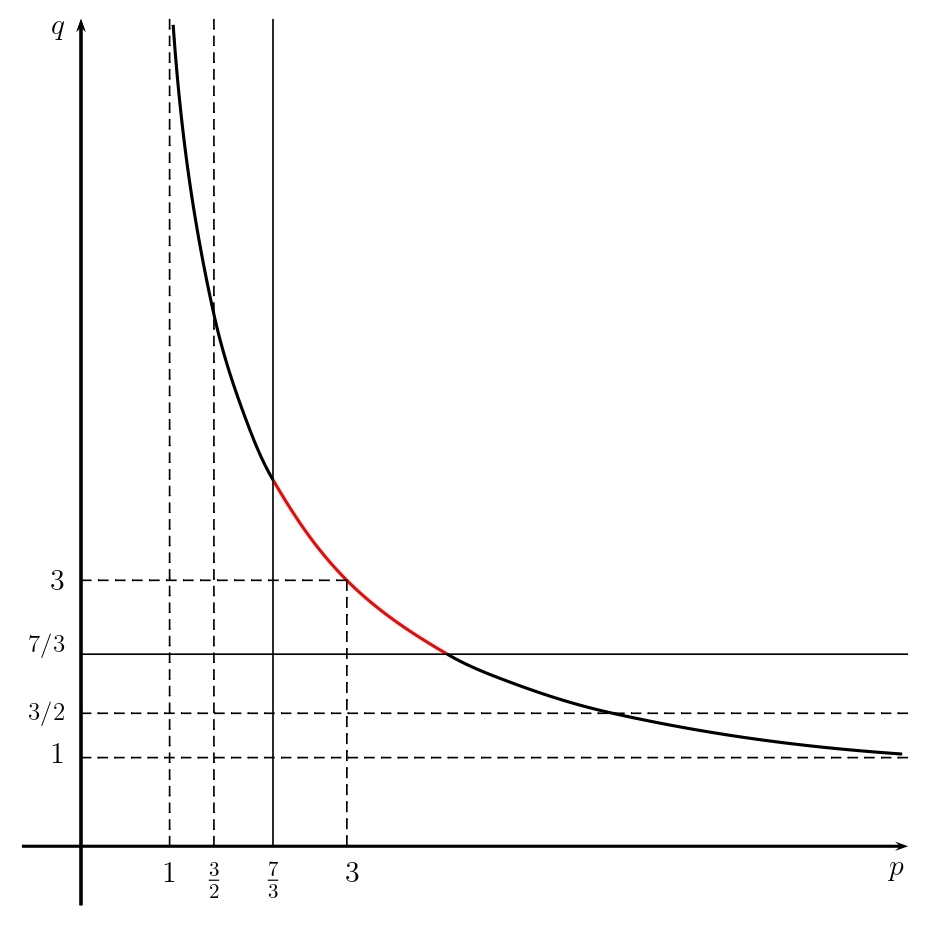}
\caption{Case $N=4$. }
\end{subfigure}%
\begin{subfigure}{0.48 \textwidth}
    \centering
  \includegraphics[width=\linewidth]{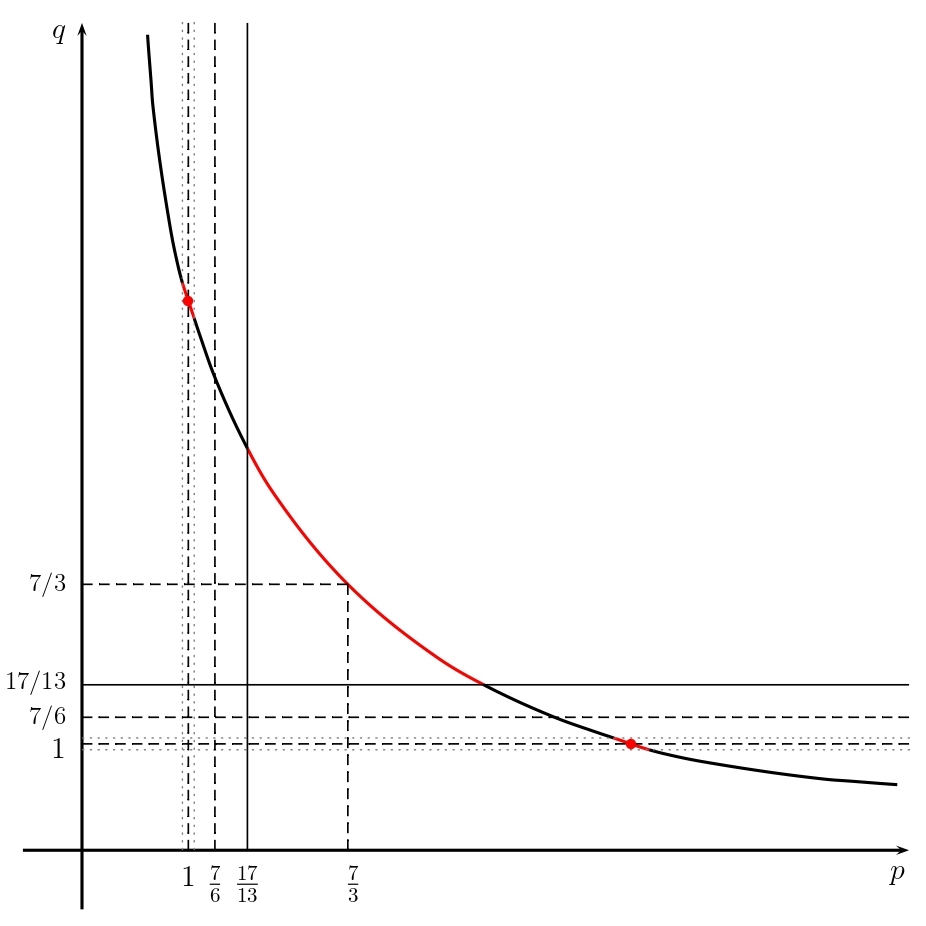}
\caption{Case $N=5$. }
\end{subfigure}
\caption{The red points on the critical hyperbola represent the values $(p, q)$ for which we prove existence of least energy solutions to \eqref{system}, see Theorem \ref{main thm}. Notice that in case $N=5$ we can also prove existence close to $p=1$ and $q=1$, see Theorem \ref{main thm1.2}. }\label{fig:56}
\end{figure}

In order to avoid the strongly indefinite character of the functional $I$ (the quadratic term $(u,v)\mapsto \int_\Omega \nabla u\cdot \nabla v$ does not have a sign), we work in a dual formulation (see Subsection \ref{subsec:variational} below). Following \cite{CK,ST2}, we provide equivalent variational characterizations for the least energy level (definition \eqref{eq:Dpq} and Proposition \ref{prop:equiv}). 

Since the problem is critical, the embeddings \eqref{eq:Sobolevembeddings} are not compact, and in general the dual functional does not satisfy the Palais-Smale condition. We prove however a compactness condition above a certain energy level (Lemma \ref{condition D}), which is based on a new class of Cherrier type inequalities, namely for every $\eps>0$ there exists $C(\eps)>0$ such that
\begin{equation*}
 \norm{u}_{\frac{N\eta}{N-2\eta}} \le  \left (\frac{2^{\frac2N}}{S} +\varepsilon\right ) \norm{\Delta u}_{\eta} + C(\varepsilon) \norm{u}_{W^{1,\eta}}, \quad \forall u \in W^{2, \eta}_\nu(\Omega)
 \end{equation*}
(see Theorem \ref{thm:Cherrier} for more details). Here we are inspired by \cite{BCN}, in which the case $\eta=2$ is shown.

\begin{figure}
 \centering 
\begin{subfigure}{0.47 \textwidth}
    \centering
 \includegraphics[width=\linewidth]{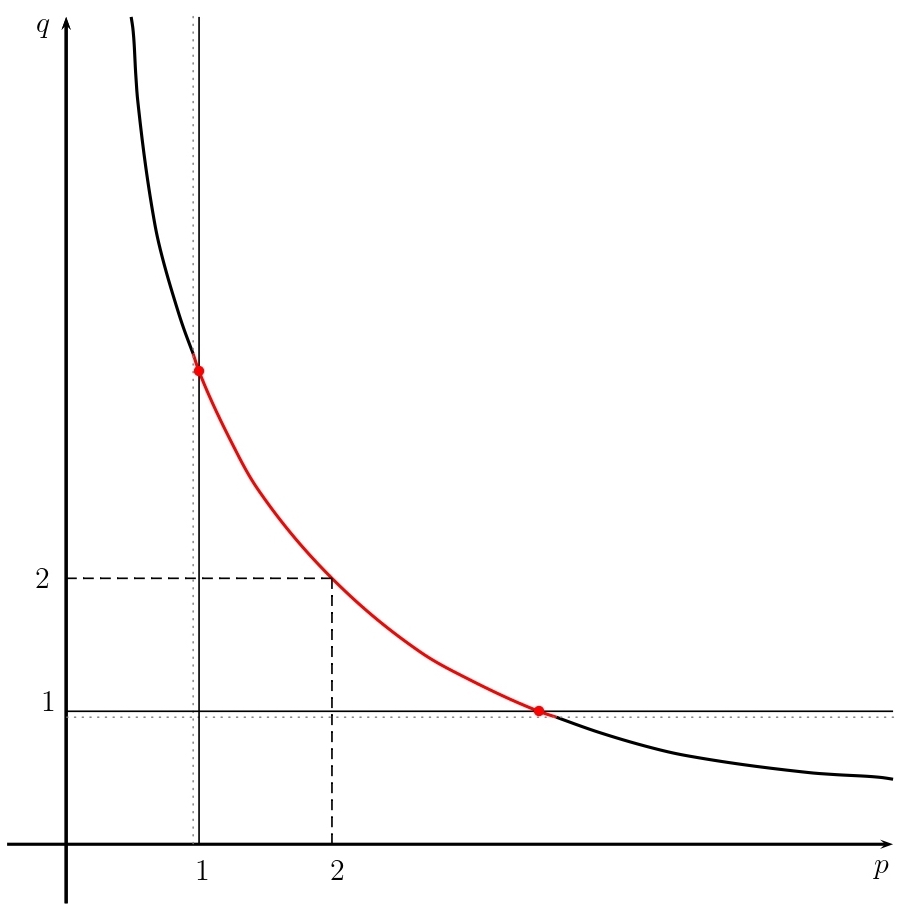}
\caption{Case $N=6$. }
\end{subfigure}%
\begin{subfigure}{0.49 \textwidth}
    \centering
  \includegraphics[width=\linewidth]{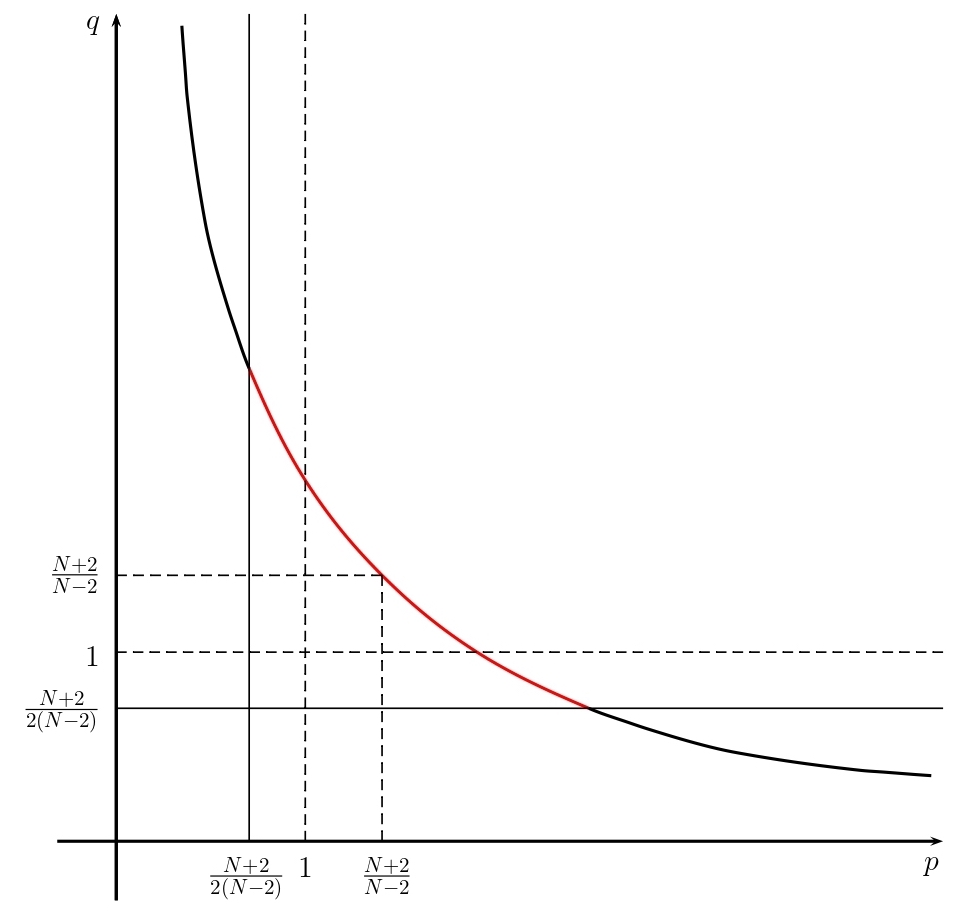}
\caption{Case $N\ge7$. }
\end{subfigure}
\caption{The red points on the critical hyperbola represent the values $(p, q)$ for which we prove existence of least energy solutions to \eqref{system}, see Theorem \ref{main thm}. Notice that in case $N=6$ we can also prove existence close to $p=1$ and $q=1$, see Theorem \ref{main thm1.2}. }\label{fig:Nge7}
\end{figure}

The way in which we check the compactness condition is where the proofs of Theorems \ref{main thm} and \ref{main thm1.2} are different. In the proof of Theorem \ref{main thm} we use a test function argument; we consider the problem in the whole space
\begin{equation}\label{eq:limitingsystem_intro}
-\Delta U = V^q  \text{ in } \R^N,\qquad
-\Delta V = U^p  \text{ in  } \R^N,
\end{equation}
which admits a family of solutions $(U_{x_0,\eps},V_{x_0,\eps})\in \mathcal{D}^{2,\frac{q+1}{q}}(\R^N)\times \mathcal{D}^{2,\frac{p+1}{p}}(\R^N)$, which are positive and radially decreasing with respect to $x_0\in \R^N$, and $\eps$ is a concentration parameter related to the scalings of the system, see Subsection \ref{sec:bubbles} for the details. Performing careful asymptotic estimates, under the assumptions of Theorem \ref{main thm} we prove the compatibility condition; these estimates are not at all straightforward mainly by two aspects: on the one hand $U_{x_0,\eps},V_{x_0,\eps}$ do not have an explicit expression and we only have access to their decay at infinity, and on the other hand we are dealing with a dual formulation. 

\begin{remark}
For convenience of the reader, we would also like to stress why we need to ask the  technical conditions \textit{(i)}, \textit{(ii)} and \textit{(iii)} in Theorem \ref{main thm}: condition $p,q>\frac{N+2}{2(N-2)}$ is needed right at the end of the proof of Lemma \ref{step3}. The hypothesis in \textit{(ii)} is equivalent to $p, q < N=5$, and it is more restrictive than $p,q>\frac{N+2}{2(N-2)}=\frac{7}{6}$. 
Also, the hypothesis in \textit{(iii)} is equivalent to $p, q < N=4$, and it is more restrictive than $p,q>\frac{N+2}{2(N-2)}=\frac{3}{2}$.  They are needed in the proof of Lemma \ref{step5}. It is an open question whether these conditions can be removed, and if there exist least energy nodal solutions for \emph{all} $p,q$ on the critical hyperbola \eqref{CH}.
\end{remark}

\begin{remark}
Notice that the case $N=3$ cannot be treated with our arguments, as for instance the hypothesis $p, q < N$, needed in the proof of Lemma \ref{step5}, is too restrictive if $N=3$, as no points on the hyperbola satisfy it. 
\end{remark}

\begin{remark}
We point out that $p=q=\frac{N+2}{N-2}$ satisfies the condition $q > \frac{N+2}{2(N-2)}$, and also the more restrictive conditions appearing in \textit{(ii)} and \textit{(iii)} if $N=4, 5$. This case (by \cite[Lemma 2.6]{ST2}) reduces to the equation with critical exponent 
\begin{equation}\label{eq:singleequationcase} - \Delta u = |u|^{\frac{4}{N-2}}u \text{  in }  \Omega,\qquad  u_\nu=0 \text{   on }  \partial \Omega, \end{equation}
and it is treated in \cite{CK} (where actually a more general operator, $Lw:=-\Delta w+\lambda w$, $\lambda\in \R$ is considered). The first part of our paper can be seen as an extension of  \cite{CK} to Lane-Emden systems. For completeness we point out that, for the single equation case \eqref{eq:singleequationcase},  the sublinear case is treated in \cite{PW}, while the subcritical superlinear case is shown in \cite[Corollary 1.4]{ST2}. The continuity of solutions with respect to $p$ is the object of \cite{ST}, while in \cite{PST} it is shown the existence of solutions in the slightly \emph{supercritical} case in some symmetric domains.
\end{remark}

In \cite{ST2} some symmetry and symmetry breaking results are proved if the domain is a ball or an annulus, and $p, q$ are subcritical. %However, when $\Omega$ is a ball and $p, q$ are on the critical hyperbola, proving regularity of  least energy solutions  is not a trivial task, as iteration techniques and Brezis-Kato type arguments do not seem to apply. 
%We then restrict ourselves to the case of $\Omega$ being an annulus centered at the origin, for which regularity up to the boundary is easier to tackle.
It is natural to wonder if such results can be extended to the critical case, and, in particular, if the least energy (nodal) solutions obtained in Theorem \ref{main thm} are radially symmetric and, if not, if there are radial solutions. 
Assume $\Omega$ satisfies
\begin{align}
\label{annulus} \Omega&=B_R(0)\setminus B_r(0), \qquad \text{ for some } 0<r<R, \qquad \text{ or }\\
\label{ball} \Omega&=B_R(0) \qquad \text{ for some } R>0.
\end{align} 
We let 
\[
c^{rad}_{p,q}:=\inf\{I(u,v):\ (u,v)\in W_{{rad}}^{2, \frac{q+1}{q}} (\Omega) \times W_{{rad}}^{2, \frac{p+1}{p}} (\Omega) \text{ strong solution of } \eqref{system}\};
\]
and we define a least energy (nodal) radial solution as a strong solution  $(u, v) \in W_{rad}^{2, \frac{q+1}{q}} (\Omega) \times W_{rad}^{2, \frac{p+1}{p}} (\Omega)$
 of \eqref{system}, such that $I(u,v)=c^{rad}_{p,q}$. 
 
We will prove the following: 
\begin{theorem}\label{thm:annulus}
We assume that $p, q$ satisfy \eqref{CH}. 
\begin{enumerate}
\item[(i)] If $\Omega$ is the annulus \eqref{annulus} then the set of least energy radial solutions of \eqref{system} is nonempty. Moreover, if $(u,v)$ is a least energy radial solution, then $u_rv_r>0$, that is, $u, v$ are strictly monotone in the radial variable, having the same monotonicity.

\item[(ii)] Moreover, let $\Omega$ as in \eqref{annulus} or \eqref{ball}, 
and let $(p,q)$ satisfy the assumptions of Theorem \ref{main thm}.
If $(u,v)$ is a least-energy nodal solution, then both $u$ and $v$ are not radially symmetric, and there exists $e^* \in \mathcal{S}^{N-1}$ such that $u, v$ are foliated Schwarz symmetric with respect to $e^*$ (see Section \ref{sec:symm} below for the definition). 
\end{enumerate}
\end{theorem}
\begin{remark}
Notice that if $p=q=2^*$, then there exist no radial solutions on the ball, whereas this is an open problem for the system, see Remark \ref{rmk:pohoz}. We also stress that continuity up to the boundary is crucial in the proof of Part \textit{(ii)} of Theorem \ref{thm:annulus}. This is easy to show for radially symmetric solutions on an annulus, whereas for a general bounded domain it is a consequence of Proposition \ref{reg2}. 
\end{remark}

%We point out that Theorem \ref{thm:annulus} is shown in the subcritical case in \cite{ST2}, also when the domain is a ball. In (i) we extend this results for annuli for all $p,q>0$, and prove (ii) for $p,q>0$ in the critical hyperbola and additional restrictions. We are only able to consider annulus and not balls by a regularity issue: if we were able to prove that least energy solutions are  continuous in $\overline \Omega$, then Theorem \ref{thm:annulus} would also hold true when $\Omega=B_R(0)$.

\smallbreak

We conclude this introduction by mentioning some related problems for systems. As said before, the subcritical case of \eqref{system} with Neumann boundary conditions is considered in \cite{ST2}; for problems where the linear operator is instead $Lw:=-\Delta w + w$, we are aware of the papers \cite{AY, BSTilli, PR, RY, Zeng}. Observe however that in the latter case the situation is much different, as for instance positive solutions are allowed. On the other hand, in the case of Dirichlet boundary conditions, the critical problem does not admit solutions in general (see \cite{Mitidieri, VanderVorst}) and concentration results have been proved close to the critical hyperbola in \cite{ChoiKim, Guerra}. Finally, we point out that the nondegeneracy of the solutions $(U_{x_0,\eps},V_{x_0,\eps})$ in \eqref{eq:limitingsystem_intro} has been recently proved in \cite{FKP}, which in particular allowed to study a Brezis-Nirenberg type problem for Lane-Emden systems and slightly subcritical systems in \cite{Pistoia, KM} (see also \cite{LFMS}) and a Coron type problem in \cite{KimCoron}.

\medskip

This paper is organized as follows. We first set notation and describe the variational formulation of the problem, and we prove Proposition \ref{reg2}. In Section \ref{sec:suff} we give a suitable compactness condition, and in Section \ref{sec:proof} we prove Theorem \ref{main thm}. Section \ref{sec: close to 1} is devoted to the study of the biharmonic equation \eqref{biharmonic}, and to the proof of Theorem \ref{main thm1.2}. In Section \ref{sec:symm} we analyze the symmetry-breaking phenomenon on radially symmetric domains, and existence of solutions with radial symmetry in the annulus. Appendix \ref{app:A}  is devoted to the proof of a Cherrier type inequality, which is crucial in the proof of the compactness condition in Section \ref{sec:suff}, whereas Appendix \ref{app:B} contains some estimates which we exploit in the proof of Theorem \ref{main thm}.

\section{Preliminaries}
\subsection{Notation and variational setting}\label{subsec:variational}

One can use various variational settings to deal with the system \eqref{system}, see for instance the survey \cite{BST}.  For our purposes, the most convenient is to use the so called \emph{dual method}, which we describe in this section.

For $s>1$, we denote the $L^s(\Omega)$ and $W^{2,s}(\Omega)$ norms by $\|\cdot\|_s$ and $\|\cdot \|_{W^{2,s}}$ respectively. Whenever the integration domain is $\R^N$, we write it explicitly and denote the $L^2(\R^N)$ and $W^{2,s}(\R^N)$ norms by $\|\cdot \|_{L^s(\R^N)}$ and $\|\cdot \|_{W^{2,s}(\R^N)}$ respectively. We define the operator $K : X^s \to W^{2, s}(\Omega)$ such that $Kh:=u$ if and only if 
\[
-\Delta u =h  \text{ in } \Omega, \qquad
u_\nu=0  \text{ on } \partial \Omega,\qquad
\int_\Omega u=0, \]
where
\[ 
X^s=\left \{ f \in L^s: \, \int_\Omega f=0 \right \}. 
\]
The fact that $K$ is well defined and continuous is a consequence of the following regularity result, which is taken from \cite[Theorem and Lemma in page 143]{Rassias} (see also \cite[Theorem 15.2]{Agmon}).
\begin{lemma}\label{lemma:regularity}
 If $s>1$, $\Omega$ be a smooth bounded domain in $\R^N$, and $h\in X^s$. Then there is a unique strong solution $u\in W^{2,s}(\Omega)$ of 
 \begin{align}\label{Nprob}
  -\Delta u = h\quad \text{ in }\Omega,\qquad \partial_\nu u=0\quad \text{ on }\partial \Omega,\qquad \int_\Omega u = 0.
  \end{align}
Moreover, there exists $C(\Omega,s)=C>0$ such that $\|u\|_{W^{2,s}}\leq C\|h\|_s.$
\end{lemma}

Also, we define $K_t : X^{\frac{t+1}{t}}(\Omega)  \rightarrow W^{2, \frac{t+1}{t}}(\Omega)$ given by 
\begin{equation} \label{eq:def_kappap}
K_t h=K h+ \kappa_t(h) \, \text{ for some } \, \kappa_t(h) \in \R \, \text{ such that } \, \int_\Omega | K_t h|^{t-1} K_t h =0. 
\end{equation}
For
\[ 
\alpha=\frac{p+1}p, \quad \beta=\frac{q+1}q, 
\]
we also define the space 
\[
X=X^\alpha \times X^\beta 
\]
and the dual functional $\Phi:X\to \R$ by
\begin{align*}
 \Phi(f, g)&=\frac p{p+1} \int_\Omega \abs{f}^{\frac{p+1}p} + \frac q{q+1}  \int_\Omega \abs{g}^{\frac{q+1}q}- \int_\Omega g Kf=\frac{1}{\alpha}\|f\|_\alpha^\alpha  + \frac{1}{\beta}\|g\|_\beta^\beta- \int_\Omega g Kf.
 \end{align*}

We also set 
\[
\gamma_1:=\frac{\beta}{\alpha+\beta}=\frac{p(q+1)}{2pq+p+q},\quad \gamma_2:= \frac{\alpha}{\alpha+\beta}=\frac{q(p+1)}{2pq+p+q},\quad\gamma:=\gamma_1 \alpha =\gamma_2 \beta=\frac{(p+1)(q+1)}{2pq+p+q},
\] 
which are such that
\[
\gamma_1+\gamma_2=1 \quad \text{ and }\quad  \frac{1}{\alpha}+\frac{1}{\beta}=\frac{1}{\gamma}.
\]
For future reference, we also point out that
\begin{equation}\label{eq:exponent_of_scaling}
\frac{1}{p+1}+\frac{1}{q+1}=\frac{N-2}{N}\implies \frac{(p+1)(q+1)}{pq-1}=\frac{N}{2}
\end{equation} 
and that a direct consequence of these definitions and Young's inequality is that:
\begin{equation}\label{Young}
(\|f\|_\alpha \|g\|_\beta)^\gamma \leq \gamma_1 \|f\|_\alpha^\alpha+\gamma_2 \|g\|_\beta^\beta\qquad \text{ for every } (f,g)\in X.
\end{equation}

Let \[ L_{p,q}=\inf \{ \Phi(f, g): (f, g) \in X \setminus \{(0, 0)\}, \quad \Phi'(f, g)=0 \}. \]
\begin{lemma}\label{lemma:same}
 Let $(f,g)\in X$ be a critical point of $\phi$ and let $(u,v):=(K_{p} g,K_{q} f)=(| f|^{\frac1p- 1}  f,|g|^{\frac1q- 1} g)$.  Then:
 \begin{enumerate}
 \item $(u,v)\in W^{2, \frac{q+1}{q}} (\Omega)  \times W^{2, \frac{p+1}{p}} (\Omega) $, and it is a strong solution of \eqref{system};
 %\item $u,v\in C^{2,\eps}(\overline \Omega)$ for some $\eps\in (0,1)$, hence $(u,v)$ is a classical solution of \eqref{system};
\item $\displaystyle \phi(f,g)=I(u,v):=\int_\Omega \nabla u\cdot \nabla v - \frac{1}{p+1}|u|^{p+1}-\frac{1}{q+1}|v|^{q+1}\, dx$.
\end{enumerate}
In particular, $(f,g)$ achieves $L_{p,q}$ if, and only if, $(u,v)$ is a least energy solution of \eqref{system}, and
\[
c_{p,q}=L_{p,q}.
\]
\end{lemma}
\begin{proof}
Item 1. follows by reasoning exactly as in from in \cite[Lemma 2.3]{ST2}, while Item 2. follows from  \cite[Lemma 2.5]{ST2}.
\end{proof}

Therefore our goal, from now on, is to show that $L_{p,q}$ is achieved. In order to show it, we use an equivalent variation formulation. Let us define
\begin{equation}\label{eq:Dpq}
D_{p, q} =\sup \left \{ \int_\Omega f Kg: \quad  (f, g) \in X, \quad \gamma_1 \norm{f}_\alpha^\alpha + \gamma_2 \norm{g}_\beta^\beta =1 \right \}
\end{equation}
Observe that $D_{p,q}<\infty$ since for every $(f,g)\in X$ such that $\gamma_1 \norm{f}_\alpha^\alpha + \gamma_2 \norm{g}_\beta^\beta =1$ we have 
\[
|\int_\Omega fKg| \leq \|f\|_\alpha\|K g\|_{\alpha^{\prime}} \leq C_1\|f\|_\alpha\|K g\|_{W^{2, \beta}} \leq C_2\|f\|_\alpha\|g\|_\beta \leq C_2 (\gamma_1 \|f\|_\alpha^\alpha+\gamma_2 \|g\|_\beta^\beta)^\frac{1}{\gamma}=C_2
\]
where we have used \eqref{eq:Sobolevembeddings}, Lemma \ref{lemma:regularity} and \eqref{Young}.

This level has the following equivalent characterizations, which are useful later on.
\begin{lemma}\label{lemma:equivalent_def_D}We have
\begin{align}
\label{def D eq 1}
 D_{p,q}				&=\sup_{(f,g)\in X\setminus \{(0,0)\}} \frac{\displaystyle\int_\Omega fKg}{\left(\gamma_1 \norm{f}_\alpha^\alpha + \gamma_2 \norm{g}_\beta^\beta\right)^{\frac{1}{\gamma}}}\\
 \label{def D eq 2}
				&=\mathop{\sup_{(f,g)\in X}}_{f,g\neq 0}\frac{\displaystyle \int_\Omega f Kg}{\|f\|_\alpha \|g\|_\beta}.
 \end{align}
 \end{lemma}
 \begin{proof}
 Proof of \eqref{def D eq 1}: given $(f,g)\in X\setminus\{(0,0)\}$, take
 \[
 \tilde f:=\frac{f}{(\gamma_1\|f\|_\alpha^\alpha+\gamma_2\|g\|_\beta^\beta)^\frac{1}{\alpha}},\qquad  \tilde g:=\frac{g}{(\gamma_1\|f\|_\alpha^\alpha+\gamma_2\|g\|_\beta^\beta)^\frac{1}{\beta}},
 \]
 Then $\gamma_1\|\tilde f\|_\alpha^\alpha+\gamma_2\|\tilde g\|_\beta^\beta=1$, and
 \[
 D_{p,q}\geq \int_\Omega \tilde fK \tilde g=\frac{\displaystyle \int_\Omega fKg}{\left(\gamma_1 \norm{f}_\alpha^\alpha + \gamma_2 \norm{g}_\beta^\beta\right)^{\frac{1}{\gamma}}}.
 \]
 Taking the supremum in $f,g$, we obtain one inequality in \eqref{def D eq 1}. Conversely, given $(\bar f,\bar g)\in X$ such that $\gamma_1 \norm{\bar f}_\alpha^\alpha + \gamma_2 \norm{\bar g}_\beta^\beta =1$, we have
 \[
 \int_\Omega \bar f K \bar g=\frac{\displaystyle \int_\Omega \bar f K \bar g}{\left(\gamma_1 \norm{\bar f}_\alpha^\alpha + \gamma_2 \norm{\bar g} _\beta^\beta\right)^{\frac{1}{\gamma}}}\leq \sup_{(f,g)\in X\setminus \{(0,0)\}}\frac{\displaystyle \int_\Omega fKg}{\left(\gamma_1 \norm{f}_\alpha^\alpha + \gamma_2 \norm{g}_\beta^\beta\right)^{\frac{1}{\gamma}}}.
 \]
 
 \smallbreak
 
  Proof of \eqref{def D eq 2}: Being a positive supremum, clearly we may take in the characterization of $D_{p,q}$ only pairs $(f,g)$ such that $f,g\neq 0$. Young's inequality \eqref{Young} immediately shows that 
\[
 D_{p,q}	\leq \sup_{(f,g)\in X\setminus \{(0,0)\}}\frac{\displaystyle \int_\Omega f Kg}{\|f\|_\alpha \|g\|_\beta}.
\]
To prove the inverse inequality, observe that if $p=q$, then $\alpha=\beta$ and it is obvious. Thus we focus on the case $p\neq q$. We use as test function $(tf,tg)$ for $t>0$ and $(f,g)\in X$, $f,g\neq 0$, and $t>0$, obtaining:
\[
\frac{t^{2\gamma} \left(\displaystyle \int_\Omega fKg\right)^\gamma}{\gamma_1 t^\alpha \|  f\|_\alpha^\alpha + \gamma_2 t^\beta\norm{ g}_\beta^\beta}\leq (D_{p,q})^\gamma
\]
or, equivalently,
\[
\left(\frac{\int_\Omega fKg}{D_{p,q}}\right)^\gamma\leq \gamma_1 t^{\alpha(1-2\gamma_1)} \| f\|_\alpha^{\alpha} + \gamma_2 t^{\beta(1 -2\gamma_2)} \norm{g}_\beta^{\beta}=:F(t)
\]
Recall that $\gamma_1+\gamma_2=1$ and that (since $p\neq q$), $\gamma_1,\gamma_2\neq 1/2$. Therefore $(1-2\gamma_1)(1-2\gamma_2)<0$, and $F$ has a strict local minimum attained at:
\[
t_*:=\left(\frac{\|f\|_\alpha^\alpha}{\|g\|_\beta^\beta}\right)^\frac{1}{\beta-\alpha},\quad \text{ with } F(t_*)=\|f\|_\alpha^\gamma\|g\|_\beta^\gamma.
\]
This concludes the proof.
 \end{proof}

\subsection{Sobolev constants}\label{sec:bubbles}
For $\eta>1$, we denote by $\mathcal{D}^{2, \eta}(\R^N)$ the completion of $C^\infty_c(\R^N)$ with respect to the norm $\|\Delta u\|_{L^\eta(\R^N)}$. 
We define 
\begin{equation}\label{eq:Spq}
S_{p,q}:= \inf\left\{\|\Delta u\|_{L^\beta(\R^N)}:\ u\in \mathcal{D}^{2,\frac{q+1}{q}},\ \|u\|_{L^{p+1}(\R^N)}=1 \right\},
\end{equation} 
which is the best Sobolev constant for the embedding $\mathcal{D}^{2, {\frac{q+1}{q}}}(\R^N) \hookrightarrow L^{p+1}(\R^N)$, namely the best constant such that one has
\[  S_{p, q} \norm{u}_{L^{p+1}(\R^N)} \le \norm{\Delta u}_{L^\beta(\R^N)} \qquad \text{ for any $u \in \mathcal{D}^{2, {\frac{q+1}{q}}}(\R^N)$}. \]

It is known \cite[Corollary I.2, p. 165]{Lions} that $S_{p,q}$ is achieved. In particular, using also the relation \eqref{eq:exponent_of_scaling}, there exists $U$ nonnegative and radially decreasing such that
\begin{equation}\label{eq:Spq_achieved}
S_{p,q}^\frac N2 =\norm{\Delta U}_{L^\beta(\R^N)}^{\beta} = \norm{U}_{L^{p+1}(\R^N)}^{p+1}
\end{equation}
and
\[
-\Delta ((-\Delta U)^{\frac1q}) = U^p \text{ in } \R^N.
\]
Therefore, if $V:=(-\Delta U)^{\frac1q}$, then
\begin{equation}\label{eq:limitingsystem}
-\Delta U = V^q  \text{ in } \R^N,\qquad
-\Delta V = U^p  \text{ in  } \R^N.
\end{equation}
One can take $U,V$ to be positive and radially decreasing. Since $\norm{U}_{L^{p+1}(\R^N)}^{p+1}=\norm{V}_{L^{q+1}(\R^N)}^{q+1}$, then  $S_{p,q}=S_{q,p}$, the best constant in the embedding $\mathcal{D}^{2, {\frac{p+1}{p}}}(\R^N) \hookrightarrow L^{q+1}(\R^N)$. Observe that, by the homogeneity of the nonlinearities and since $p,q$ satisfy \eqref{CH}, then we have a whole family of solutions to \eqref{eq:limitingsystem}, namely
\begin{equation}\label{eq:familybubbles}
\begin{cases}
U_{\varepsilon, x_0} (x)= \eps^{-\frac{2(q+1)}{pq-1}}U\left(\frac{x-x_0}{\eps}\right)=\varepsilon^{-\frac N{p+1}} U \left(\frac{x-x_0}{\varepsilon} \right),\\V_{\varepsilon, x_0} (x)=\eps^{-\frac{2(p+1)}{pq-1}}V\left(\frac{x-x_0}{\eps}\right)=\varepsilon^{-\frac N{q+1}} V \left(\frac{x-x_0}{\varepsilon} \right) 
\end{cases}
\end{equation}
 and
  \[ S_{p, q}^{\frac N2} =\norm{\Delta U_{\eps, x_0}}_{L^\beta(\R^N)}^{\beta} = \norm{U_{\eps, x_0}}_{L^{p+1}(\R^N)}^{p+1}. \]
It is proved in \cite{HV} that these are all the positive solutions of \eqref{eq:limitingsystem}. We just mention that there exist also sign-changing solutions to \eqref{eq:limitingsystem}, see \cite{ClappSaldana}. Without loss of generality, we assume that 
\begin{equation*}
V(0)=1,\qquad \text{ and that } \qquad q\leq \frac{N+2}{N-2}.
\end{equation*}

In the following we do a slight abuse of notation and use $w(|x|)=w(x)$ for a radial function $w$. By \cite[Theorem 2]{HV}, we have that there exists $a,b>0$ such that
\begin{equation} \label{eq:decayestimate}
\lim _{r \rightarrow \infty} r^{N-2} V(r)=a,\qquad \begin{cases}
\lim _{r \rightarrow \infty} r^{N-2} U(r)=b  & \text { if } q>\frac{N}{N-2} \\ 
\lim _{r \rightarrow \infty} \frac{r^{N-2}}{\log r} U(r)=b  & \text { if } q=\frac{N}{N-2} \\ 
\lim _{r \rightarrow \infty} r^{q(N-2)-2} U(r)=b  & \text { if } q<\frac{N}{N-2}
\end{cases}
\end{equation}
Based on these estimates, in \cite[pp. 2360--2362]{HMV} it is shown that, as $\eps\to 0^+$:
\begin{equation}\label{eq:norm_estimates}
\|V_\eps\|_{L^1(\R^N)}\approx \|U_\eps^p\|_{L^1(\R^N)}\approx \eps^\frac{N}{p+1},\quad \|U_\eps\|_{L^1(\R^N)}\approx \|V_\eps^q\|_{L^1(\R^N)} \approx \begin{cases}
\eps^\frac{N}{q+1} & \text{ if } \frac{N}{N-2}<q\leq \frac{N+2}{N-2},\\
\eps^\frac{N(N-2)}{2(N-1)}|\log \eps| & \text{ if } q=\frac{N}{N-2},\\
\eps^\frac{qN}{p+1} & \text{ if } q<\frac{N}{N-2} 
\end{cases}
\end{equation}
where $f \approx g$ means that the quotient there exists $C>1$ such that $1/C\leq f/g\leq C$. Notice that in \cite{HMV} it is also assumed $p, q>1$, however, an inspection of the proof shows that the estimates \eqref{eq:norm_estimates} are true without this restriction.
We also point out, for future reference, that from  \cite[equations (3.22)--(3.24)]{HV} we have
\begin{equation}\label{eq:decayestimates2}
rU'(r)\approx U(r),\qquad rV'(r)\approx V(r).
\end{equation}

We also need to introduce the best Sobolev constant $S_*$ for the embedding $\mathcal{D}^{1, 2}(\R^N) \hookrightarrow L^{2^*}(\R^N)$, namely
\[ S_* \|u\|_{L^{2^*}(\R^N)}^2 \leq \|\nabla u\|_{L^2(\R^N)}^2. \]
Associated with it we have the family of \emph{bubbles}:
\[
U^*_{x_0,\eps}(x):=\eps^{-\frac{N-2}{2}}U^*\left(\frac{x-x_0}{\eps}\right),\quad \text{ with } \quad  U^*(x):=\frac{(N(N-2))^\frac{N-2}{4}}{(1+|x|^2)^\frac{N-2}{2}}
\] 
which are the unique solutions in $\mathcal{D}^{1,2}(\R^N)$ of 
\[ -\Delta U^*_{\eps, x_0} = (U^*_{\eps, x_0}) ^{2^*-1}  \text{ in } \R^N. \]
They satisfy
\[ \norm{ \nabla U^*_{\eps, x_0} }_{L^2(\R^N)}^2 = \norm{U^*_{\eps, x_0} }_{L^{2^*}(\R^N)}^{2^*}=S_*^{\frac N2}. \]
\begin{remark}
Observe that, for 
\[
p=q=2^*-1=\frac{N+2}{N-2}, \quad  \text{ we have } U^*=U=V,\quad \text{ and so } S_{\frac{N+2}{N-2},\frac{N+2}{N-2}}=S_*. 
\] 
%The constant $S$ in \cite{CK} is $S^*$, the constant $S$ in \cite{BCN} is $S_{p, q}^2$, with $p=1$ and $q= \frac{N+4}{N-4}$.
\end{remark}
\subsection{Equivalence between $D_{p,q}$ and the least energy nodal level}

\begin{proposition}\label{prop:equiv}
Assume $D_{p, q}$ is attained at $(f, g)$. Then $(u, v)=((D_{p,q})^{-q \frac{p+1}{pq-1}} K_p g, (D_{p,q})^{-p \frac{q+1}{pq-1}} K_q f)$ is a least energy solution for \eqref{system}, and
\begin{equation} \label{eq:relation_energylevels}
D_{p,q}^{-\frac{N}{2}}=\frac{N}{2} L_{p, q}=\frac{N}{2} c_{p, q}. 
%D_{p,q}^{-\frac{(p+1)(q+1)}{pq-1}}=\frac{(p+1)(q+1)}{pq-1} L_{p, q}=\frac{(p+1)(q+1)}{pq-1} c_{p, q}. 
\end{equation}
\end{proposition}
\begin{proof}

Assume that $D_{p,q}$ is attained. Then there exist functions $(f,g) \in X$ such that $\gamma_1 \norm{f}_\alpha^\alpha + \gamma_2 \norm{g}_\beta^\beta=1$ and $\int_\Omega f Kg =D_{p, q}$. 
We now use Lagrange multipliers and the definition of $\gamma$ to conclude that there exists $\lambda \in \R$ such that 
\begin{equation}\label{lagrange} 
\int_\Omega \varphi K g= \lambda \gamma \int_\Omega |f|^{\frac1p- 1} f \varphi,\qquad
\int_\Omega \psi K f= \lambda \gamma \int_\Omega |g|^{\frac1q -1} g \psi
 \end{equation}
for any $(\varphi, \psi) \in X$. Now let $\tilde \varphi \in L^\alpha(\Omega)$, $\tilde \psi \in L^\beta(\Omega)$. Define
\[ \varphi= \tilde \varphi - \frac1{|\Omega|} \int_\Omega \tilde \varphi, \quad  \psi= \tilde \psi - \frac1{|\Omega|} \int_\Omega \tilde \psi. \]
Notice that $( \varphi,  \psi) \in X$. Then
\[ \int_\Omega \tilde \varphi Kg- \frac1{|\Omega|} \int_\Omega Kg  \int_\Omega  \tilde \varphi=\lambda \gamma \int_\Omega |f|^{\frac1p- 1} f \tilde \varphi  - \lambda \gamma \frac1{|\Omega|} \int_\Omega |f|^{\frac1p- 1} f\int_\Omega \tilde  \varphi. \] 
This implies, recalling that $\int_{\Omega} K g=0$ and $\tilde \varphi$ is arbitrary,  
\[ Kg + \kappa=\lambda \gamma |f|^{\frac1p- 1} f , \quad \text{ with } \quad
 \kappa=\lambda \gamma \frac1{|\Omega|}\int_\Omega |f|^{\frac1p- 1} f. \]
Since 
\[
\int_\Omega |Kg + \kappa|^{p-1}(Kg + \kappa)=(\lambda \gamma)^p\int_\Omega f=0,
\]
then $\kappa=\kappa_p(Kg)$ (recall \eqref{eq:def_kappap}), and \[K_p g=Kg+\kappa=\lambda \gamma |f|^{\frac1p- 1} f. \]
Similarly,
\[K_q f=\lambda \gamma |g|^{\frac1q- 1} g. \]
Also, by \eqref{lagrange},  
\[ \begin{cases}
\int_\Omega f Kg= \lambda \gamma \norm{f}_\alpha^\alpha \\
\int_\Omega gKf=\lambda \gamma \norm{g}_\beta^\beta.
\end{cases},\quad \text{ so that } \quad D_{p,q}=\int_\Omega f Kg =\int_\Omega gKf=\lambda \gamma \norm{f}_\alpha^\alpha=\lambda \gamma \norm{g}_\beta^\beta.
 \]
Hence, since $\gamma_1+\gamma_2=1$,
\[ D_{p, q}= (\gamma_1+\gamma_2)D_{p, q} =\lambda \gamma (\gamma_1 \norm{f}_\alpha^\alpha + \gamma_2 \norm{g}_\beta^\beta)^{\frac{\gamma-1}{\gamma}}=\lambda \gamma, \]
from which we conclude
\[
K_p g= D_{p, q} |f|^{\frac1p- 1} f \qquad
K_q f= D_{p, q} |g|^{\frac1q- 1} g. 
 \]

Now consider the functions $\tilde f=(D_{p,q})^s f$, and $\tilde g=(D_{p,q})^t g$, with $s=-p\frac{q+1}{pq-1}$ and $t=-q\frac{p+1}{pq-1}$. Then, since $K_p,K_q$ are homogeneous of degree 1,
\[
K_p \tilde g=  |\tilde f|^{\frac1p- 1} \tilde f \qquad
K_q \tilde f=  |\tilde g|^{\frac1q- 1}\tilde g. 
\]
In particular, $\Phi'(\tilde f, \tilde g)=0$. 

One has
\begin{align}\label{inequality 1 L}
L_{p,q} \le \Phi(\tilde f, \tilde g) &= \frac{p}{p+1} \|\tilde f\|_\alpha^\alpha +\frac{q}{q+1} \norm{\tilde g}_\beta^\beta - \int_\Omega \tilde f K \tilde g \\
\nonumber&=\frac{pq-1}{(p+1)(q+1)} \int_\Omega \tilde fK \tilde g=\frac{pq-1}{(p+1)(q+1)}  D_{p,q}^{s+t+1} =\frac{pq-1}{(p+1)(q+1)} D_{p,q}^{-\frac{(p+1)(q+1)}{pq-1}}. 
\end{align}

On the other hand, taking $(f,g) \in X\setminus \{(0,0)\}$ any critical point of $\Phi$, then necessarily $f,g\not\equiv 0$ and  $(\frac{f}{\|f\|_\alpha},\frac{g}{\|g\|_\beta})$ is a test function for $D_{p,q}$, and
\begin{align*}
D_{p,q } &\ge \int_\Omega \frac{ f}{\| f\|_\alpha} K \frac{g}{\norm{ g}_\beta} = \frac{\int_\Omega fK  g}{(\int_\Omega  f K  g)^{\frac1\alpha} (\int_\Omega  f K  g)^{\frac1\beta} } \\
&= \left(\int_\Omega  f K  g\right)^{-\frac{pq-1}{(p+1)(q+1)}} = \left( \frac{(p+1)(q+1)}{pq-1} \Phi( f,  g) \right)^{- \frac{pq-1}{(p+1)(q+1)}}.
\end{align*} Taking the infimum in $(f,g)$ and comparing it with the opposite inequality, we obtain the identity:
from which we conclude \eqref{eq:relation_energylevels}. Also, by \eqref{inequality 1 L} we have that $L_{p, q}$ is attained at $(\tilde f, \tilde g)$.  We can now conclude from Lemma \ref{lemma:same} and using the identity \eqref{eq:exponent_of_scaling}.
\end{proof}

\subsection{Regularity of solutions}
In order to prove Proposition \ref{reg2}, we recall the following result
\begin{proposition}[Proposition 9 in \cite{DRW}, Lemma 3.1 in \cite{reywei}]\label{green}
Let $f $ be such that $\int_\Omega f=0$. Let $G$ be the Green function for
\begin{equation}\label{eq:green} -\Delta u = f \quad \text{ in } \Omega, \qquad  \partial_\nu u=0 \quad \text{ in } \partial \Omega, \end{equation}
namely $G : \Omega \times \Omega \setminus \{ (x, x): x \in \Omega \} \to \R$ is such that 
\begin{itemize}
\item[(i)] $G(x, \cdot) \in L^1(\Omega)$, 
\item[(ii)] $\int_\Omega G(x, y) \, dy=0$, 
\item[(iii)] for any $u $ solution to \eqref{eq:green} one has 
\[ u(x) - \frac{1}{|\Omega|} 	\int_\Omega u = \int_\Omega G(x, y) f(y) \, dy. \]
\end{itemize}
Then
\[ \abs{ G(x, y)} \le C(\Omega) \abs{x-y}^{2-N}, \]
 for all $x, y \in \Omega$, $x \ne y$. 
\end{proposition}
\begin{proof}[Proof of Proposition \ref{reg2}]
We take inspiration from \cite[Theorem 1.3]{ChenLi}, see also \cite[Lemma B.1]{KP}. 
Let us take $r > \frac{N}{N-2}, \frac{N}{Nq-2q-2}$ fixed but arbitrary. Also, we set $s:=\frac{Nrq}{N+2r}$ and 
\[ t:= \frac{q(p+1)}{pq-1} \frac{N+2s}{Ns} >0, \quad t':=rq \frac{N+2s}{Ns} >0. \]
We preliminary notice that 
\begin{equation}\label{s>}
s > \frac{N}{N-2}
\end{equation}
and
\begin{equation}\label{exp holder}
\frac 1 t + \frac 1 {t'}=1.
\end{equation}
Indeed, \eqref{s>} immediately follows recalling $r>\frac{N}{Nq-2q-2}$, and using $q > \frac{2}{N-2}$. 
As for \eqref{exp holder}, one observes that it is equivalent to prove
\[ \frac{N+2r+2rq}{Nrq}=\frac{N+2s}{Ns}= \frac{pq-1}{q(p+1)} + \frac {1}{rq}, \]
namely
\[ \frac {2(q+1)}{N}=\frac{pq-1}{p+1} \]
which is satisfied if and only if $p, q$ belong to the critical hyperbola \eqref{CH}. 

Now, let us define 
\[ u_L(x):=\begin{cases} u(x) &\text{ if } |u(x)| > L \\
0 &\text{ if } |u(x)| \le L. \end{cases} \]
We take $f \in L^r(\Omega)$, and define the operator $T_L: L^r(\Omega) \to L^r(\Omega)$ such that 
\[ T_L(f)(x):=\int_\Omega G(x, y) \left \{ \abs{\int_\Omega G(y, z) h(z) \, dz}^{q-1} \left( \int_\Omega G(y, z) h(z) \, dz \right) \right \}\, dy \]
where
\[ h(z):= |u_L(z)|^{p-\frac 1 q}  |f(z)|^{ \frac 1 q-1} f(z). \]
%\textcolor{red}{Is it well defined? It is a consequence of estimates below}
By applying Proposition \ref{green} and the Hardy--Littlewood--Sobolev inequality (recall that $r > \frac{N}{N-2}$), one has
\[ \norm{T_L f}_r \le C \norm{ \left( \int_\Omega |G(y,z)| |u_L(z)|^{p-\frac 1 q} |f(z)|^{\frac 1 q} \, dz \right)^q}_{\frac{Nr}{N+2r}} = C \norm{\int_\Omega |G(y,z)| |u_L(z)|^{p-\frac 1 q} |f(z)|^{\frac 1 q} \, dz}_s^{q}, \]
where $C>0$ denotes a positive constant.
We now apply once again Proposition \ref{green} and Hardy--Littlewood--Sobolev inequality, this time using $s > \frac{N}{N-2}$, to get
\[ \norm{\int_\Omega |G(y,z)| |u_L(z)|^{p-\frac 1 q} |f(z)|^{\frac 1 q} \, dz}_s^{q} \le C \norm{ |u_L|^{p-\frac 1 q} |f|^{\frac 1 q}}_{\frac{Ns}{N+2s}}^{q}. \]
The H\"older inequality with exponents $t, t'$ yields
\[ \norm{ |u_L|^{p-\frac 1 q} |f|^{\frac 1 q}}_{\frac{Ns}{N+2s}}^{q} \le \norm{|u_L|^{p-\frac 1q}}_{\frac{Nst}{N+2s}}^{q} \norm{|f|^{\frac 1q}}_{\frac{Nst'}{N+2s}}^{q}= \norm{u_L}_{\frac{pq-1}{q} \frac{Nst}{N+2s}}^{pq-1} \norm{f}_{\frac{t'}{q} \frac{Ns}{N+2s}} =\norm{u_L}_{p+1}^{pq-1} \norm{f}_r,  \]
from which we deduce
\[ \norm{T_L f}_r \le C \norm{u_L}_{p+1}^{pq-1} \norm{f}_r. \]
Since $u \in L^{p+1}(\Omega)$, we can choose $L$ large enough such that $\norm{u_L}_{p+1}$ is small enough. This implies that $T_L(f)$ is a contraction mapping from $L^r(\Omega)$ into itself. Moreover, 
\[ u_L=T_L (u_L) + F, \quad \text{ where } F \text{ is uniformly bounded.} \]
% \int u è costante, irrilevante
 Thus, by the contraction mapping theorem, see Theorem 1 in \cite{ChenJinLiLim}, one has
 $u_L \in L^r(\Omega)$. Since $r$ is arbitrary, we conclude that $u \in L^\infty(\Omega)$. 
 
 Notice that the same arguments give $v \in L^\infty(\Omega)$. Thus, $|u|^{p-1} u \in  L^{r/p}(\Omega)$ for any $r >p$, whereas $|v|^{q-1} v \in L^{r/q}(\Omega)$ for any $r >q$. This immediately gives $v \in W^{2, r}(\Omega)$ for any $r >p$, and $u \in W^{2, r}(\Omega)$ for any $r >q$. We now apply Sobolev embeddings to get $u, v \in C^{0, \gamma}(\overline \Omega)$ for any $\gamma \in (0, 1)$. Schauder regularity theory \cite[Theorem 6.31]{GT} yields $(u, v) \in C^{2, \zeta}(\overline \Omega) \times C^{2, \eta}(\overline \Omega)$, where $\zeta, \eta$ are as in the statement. 
  \end{proof}
\begin{remark}
Actually if $p, q \ge 1$ one can conclude $u, v \in C^\infty(\overline \Omega)$ by a bootstrap argument. 
Notice that the proof above also shows regularity for the Dirichlet system 
\[ \begin{cases}
-\Delta u= \abs{v}^{q-1}v & \text{ in } \Omega\\
-\Delta v= \abs{u}^{p-1}u & \text{ in } \Omega\\
u=v=0 & \text{ on } \partial \Omega,
\end{cases}  \]
with $p, q$ on the critical hyperbola \eqref{CH}. 
%With the additional assumption $p, q \ge 1$, this is essentially contained in \cite[Theorem 1.3]{ChenLi} and \cite[Lemma B.1]{KP}. 
\end{remark}

\section{A compactness condition to have $D_{p,q}$ attained}\label{sec:suff}
This Section is devoted to the proof of the following lemma. 
\begin{lemma}[Compactness condition]\label{condition D}
Let $D_{p,q}$ and $S_{p,q}$ be as in \eqref{eq:Dpq} and \eqref{eq:Spq}, respectively. If $p, q$ satisfies \eqref{CH} and  
\[  D_{p,q} > \frac{2^{2/N}}{S_{p,q} }, \]
then $D_{p,q}$ is attained.
\end{lemma}
\begin{proof}
Let us define, for any $(f, g) \in X$,
\[ F(f, g)=\int_\Omega f Kg=\int_\Omega g Kf \quad \text{ and } \quad H(f, g)=\gamma_1 \norm{f}_\alpha^\alpha + \gamma_2 \norm{g}_\beta^\beta.\]
Take a maximizing sequence $(f_k, g_k)$ for $D_{p,q}$.
By Ekeland's variational principle, we find $(\tilde f_k, \tilde g_k)\in X$ and $\lambda_k \in \R$ such that
\begin{equation}\label{cond} \gamma_1 \| \tilde f_k \|_\alpha^\alpha + \gamma_2 \norm{\tilde g_k}_\beta^\beta=1, \end{equation}
\[ 
\norm{f_k-\tilde f_k}_\alpha \to 0, \qquad \norm{g_k-\tilde g_k}_\beta \to 0,\qquad F(\tilde f_k, \tilde g_k) \to D_{p, q}, \]
and
\[ F'(\tilde f_k, \tilde g_k)-\lambda_k H'(\tilde f_k, \tilde g_k) \to 0 \text{ in } X^*. \]
Up to a subsequence, there exist $(f, g) \in X$ such that
\begin{equation}\label{eq:weakconv}
 \tilde f_k \rightharpoonup f \text{ weakly in } L^\alpha(\Omega), \quad \tilde g_k \rightharpoonup g \text{ weakly in } L^\beta(\Omega). 
 \end{equation}
Our main goal is to prove that this convergence is strong, after which it is straightforward to see that $(f,g)$ achieves $D_{p,q}$. Since the proof is long, we split it in several steps. 
\smallbreak

\noindent \textit{First step.} We show that $\tilde f_k \to f$ and $\tilde g_k \to g$ almost everywhere. 
For any $\varphi \in L^\alpha(\Omega)$, define 
\[ \tilde \varphi= \varphi - \frac1{\abs{\Omega}} \int_\Omega \varphi \in X^\alpha\]
and notice that $ \| \tilde \varphi \|_\alpha \le 2 \| \varphi \|_\alpha$.
Let $c_k:=\frac{\gamma \lambda_k}{|\Omega|}  \int_\Omega |\tilde f_k |^{\frac1p-1} \tilde f_k$,
so that 
\begin{align*} \int_\Omega \varphi K \tilde g_k +&\int_\Omega c_k \varphi - \gamma \, \lambda_k \int_\Omega |\tilde f_k |^{\frac1p-1} \tilde f_k \varphi=
\int_\Omega \tilde \varphi K \tilde g_k -\gamma\, \lambda_k \int_\Omega |\tilde f_k |^{\frac1p-1} \tilde f_k \tilde \varphi \\
&\le \norm{\partial_1 F(\tilde f_k, \tilde g_k)-\lambda_k \partial_1 H(\tilde f_k, \tilde g_k)}_{(X^{\alpha})^*} \norm{ \tilde \varphi}_\alpha  \le2 \norm{\partial_1 F(\tilde f_k, \tilde g_k)-\lambda_k \partial_1 H(\tilde f_k, \tilde g_k)}_{(X^{\alpha})^*} \norm{ \varphi}_\alpha. \end{align*}
Hence
\begin{align*}
\norm{ K \tilde g_k + c_k - \lambda_k  \gamma |\tilde f_k |^{\frac1p-1} \tilde f_k }_{p+1} &= \norm{ K \tilde g_k + c_k - \lambda_k  \gamma |\tilde f_k |^{\frac1p-1} \tilde f_k }_{(L^{\alpha})^*} \\
&= \sup_{\varphi\in L^\alpha (\Omega)\setminus \{0\}} \frac{\abs{ \int_\Omega \varphi K \tilde g_k +\int_\Omega c_k \varphi - \gamma \, \lambda_k \int_\Omega |\tilde f_k |^{\frac1p-1} \tilde f_k \varphi }}{\norm{\varphi}_\alpha} \\
&\le 2 \norm{\partial_1 F(\tilde f_k, \tilde g_k)-\lambda_k \partial_1 H(\tilde f_k, \tilde g_k)}_{(X^{\alpha})^*} \to 0.
\end{align*}
In particular, up to a subsequence,
\[
K \tilde g_k + c_k - \lambda_k  \gamma |\tilde f_k |^{\frac1p-1} \tilde f_k  \quad \text{ converges a.e. in } \Omega.
\]
All terms in this expression converge a.e. Indeed, up to subsequences:
\begin{itemize}
\item $K \tilde g_k$ converges a.e. This is because, by Lemma \ref{lemma:regularity} and \eqref{cond}, we have that $\norm{K\tilde g_k}_{W^{2,\beta}}\leq C\norm{\tilde g_k}_\beta\leq \tilde C$  for every $k$. Therefore, by the compact embeddings $W^{2,\beta}(\Omega)\hookrightarrow L^t(\Omega)$ for every $t\in [1,p+1)$, we have the existence of $w\in L^{p+1}(\Omega)$ such that
\[
K\tilde g_k \to w \quad \text{ strongly in } L^t(\Omega) \text{ for any $1\leq t<p+1$, pointwisely a.e.}.
\]
\item $\lambda_k \to  \frac1\gamma D_{p, q} \ne 0$, since we have \begin{align*} 
o(1)&=F'(\tilde f_k, \tilde g_k)\left(\gamma_1 \tilde f_k, \gamma_2 \tilde g_k\right)- \lambda_k H'(\tilde f_k, \tilde g_k)\left(\gamma_1 \tilde f_k, \gamma_2 \tilde g_k\right) \\
&=(\gamma_1+\gamma_2) D_{p, q} - \lambda_k \gamma (\gamma_1 \| \tilde f_k\|_\alpha^\alpha + \gamma_2 \norm{\tilde g_k}_\beta^\beta) + o(1)\\
&=D_{p, q} - \lambda_k \gamma + o(1), \end{align*}
(where we recall  that $\gamma=\gamma_1 \alpha=\gamma_2 \beta$ and $\gamma_1+\gamma_2=1$).
\item there exists a constant $c_0 \in \mathbb{R}$ such that, up to a subsequence   $c_k \to c_0$ since, by \eqref{cond} and the previous paragraph, we have that  $c_k$  is bounded. 
\end{itemize}
Combining all of this we deduce that $|\tilde f_k|^{\frac{1}{p}-1}\tilde f_k$ converges a.e, and so also $\tilde f_k$. Together with \eqref{eq:weakconv}, this yields that $\tilde f_k\to f$ a.e. in $\Omega$. Similarly one proves that also $\tilde g_k \to g$ a.e. 
\smallbreak

\noindent \textit{Second step.} $\tilde f_k \to f$ strongly in $L^\alpha(\Omega)$, $\tilde g_k \to g$ strongly in $L^\beta(\Omega)$.
We write 
\[ \tilde f_k = f + w_k \qquad \text{ and } \qquad \tilde g_k =g+z_k, \]
where $w_k \rightharpoonup 0$ in $L^\alpha(\Omega)$, and $z_k \rightharpoonup 0$ in $L^\beta(\Omega)$. By Br\'ezis-Lieb's Lemma, we have
\begin{align*} 
\gamma_1 \norm{f}_\alpha^\alpha + \gamma_1\norm{w_k}_\alpha^\alpha + \gamma_2 \norm{g}_\beta^\beta + \gamma_2 \norm{z_k}_\beta^\beta 					&=\gamma_1 \|\tilde f_k\|_\alpha^\alpha+\gamma_2\|\tilde g_k\|_\beta^\beta+o(1)=1+o(1).
 \end{align*}
In particular, since $\gamma<1$ (which is equivalent to $pq>1$), then 
\begin{equation}\label{eq:BL_2}
\left(\gamma_1 \norm{f}_\alpha^\alpha + \gamma_2 \norm{g}_\beta^\beta\right)^{1/\gamma} + \left( \gamma_1\norm{w_k}_\alpha^\alpha+ \gamma_2 \norm{z_k}_\beta^\beta\right)^{1/\gamma} \leq 1+o(1).
\end{equation}
Now,
\begin{align}
D_{p, q} + o(1) = \int_\Omega \tilde g_k K \tilde f_k  &= \int_\Omega g K \tilde f_k + \int_\Omega z_k K \tilde f_k = \int_\Omega g K f + \int_\Omega g K w_k + \int_\Omega z_k K f + \int_\Omega z_k K w_k \notag\\
&=\int_\Omega g K f +\int_\Omega z_k K w_k + o(1) \label{eq:ineq_aux2}.
\end{align}
By definition of $D_{p, q}$, we can estimate the right-hand side of \eqref{eq:ineq_aux2} as follows
\begin{align*} \int_\Omega g K f +\int_\Omega z_k K w_k + o(1) &\le D_{p, q} (\gamma_1 \norm{f}_\alpha^\alpha + \gamma_2 \norm{g}_\beta^\beta)^{1/\gamma} + \int_\Omega z_k K w_k+ o(1).
\end{align*}
The left-hand side of \eqref{eq:ineq_aux2} can be estimated using \eqref{eq:BL_2} as 
\begin{align*} D_{p, q}+o(1) \ge D_{p, q}(\gamma_1\norm{w_k}_\alpha^\alpha + \gamma_2 \norm{z_k}_\beta^\beta)^{1/\gamma} +D_{p, q}(\gamma_1 \norm{f}_\alpha^\alpha + \gamma_2 \norm{g}_\beta^\beta)^{1/\gamma} +o(1).
\end{align*}
From this we deduce that
\begin{equation}\label{eq:aux_compact1}
D_{p, q}  (\gamma_1\norm{w_k}_\alpha^\alpha + \gamma_2 \norm{z_k}_\beta^\beta)^{1/\gamma}  \le \int_\Omega z_k Kw_k+ o(1). 
\end{equation}
Notice that $\norm{K w_k}_{W^{1,\alpha}} \to 0$. Indeed, by continuity of $K$ (Lemma \ref{lemma:regularity}) one has $K w_k \rightharpoonup 0$ weakly in $W^{2, \alpha}(\Omega)$ and the conclusion follows by compactness of the embedding $W^{2, \alpha}(\Omega) \hookrightarrow W^{1, \alpha}(\Omega)$. Also, 
\[ K w_k \in W^{2, \alpha}_\nu(\Omega)=\{u \in W^{2, \alpha}(\Omega): u_\nu=0 \text{ on } \partial \Omega \}, \]
 by definition of $K$. Assume without loss of generality that 
 \[
 q \ge \frac{N+2}{N-2}
 \]
(the other case follows identically since $S_{p,q}=S_{q,p}$). Since $N > 2 \frac{q+1}{q}$, we can apply the Cherrier's type inequality \eqref{cherrier} with $u= K z_k$  and $\eta=\frac{q+1}{q}$ to obtain
\[  \norm{Kz_k}_{p+1}  \le \left( \frac{2^{2/N}}{S_{p, q}} + \varepsilon \right) \norm{z_k}_{\frac{q+1}{q}} + C(\varepsilon) \norm{Kz_k}_{W^{1,\frac{q+1}{q}}}, \]
where $S_{p, q}$ is the best Sobolev constant for the embedding $\mathcal{D}^{2, {\frac{q+1}{q}}}(\R^N) \hookrightarrow L^{p+1}(\R^N)$, recall \eqref{eq:Spq}. %see Theorem 2.3 in GGS

We obtain
\begin{align}
\int_\Omega z_k Kw_k =\int_\Omega w_k Kz_k &\le \norm{w_k}_\alpha \norm{K z_k}_{p+1}  \le \norm{w_k}_\alpha \left( \frac{2^{2/N}}{S_{p, q}} + \varepsilon \right) \norm{ z_k}_{\frac{q+1}{q}}+o(1) \notag\\
								&=\left( \frac{2^{2/N}}{S_{p, q}} + \varepsilon \right) \norm{w_k}_\alpha \norm{z_k}_\beta + o(1) \notag\\
								&\le   \left( \frac{2^{2/N}}{S_{p, q}} + \varepsilon \right) (\gamma_1\norm{w_k}_\alpha^\alpha + \gamma_2 \norm{z_k}_\beta^\beta)^{1/\gamma }
+o(1), \label{eq:aux_compact2}
\end{align}
where in the last inequality we used \eqref{Young} (a consequence of Young's inequality).
Combining  \eqref{eq:aux_compact1} with \eqref{eq:aux_compact2}  yields
\begin{equation}\label{1est}  (\gamma_1\norm{w_k}_\alpha^\alpha + \gamma_2 \norm{z_k}_\beta^\beta)^{1/\gamma} \left(D_{p, q} -\frac{2^{2/N}}{S_{p,q}} - \varepsilon \right)  \le  
o(1). \end{equation}
Hence,  if  
\[  1 > \frac{2^{2/N}}{S_{p,q} D_{p,q}} \]
holds true, then we have, by taking $\varepsilon$ small enough, that 
\[  \norm{w_k}_\alpha\to 0 \quad \text{ and } \quad  \norm{z_k}_\beta \to 0, \] whence  $\tilde f_k \to f$ strongly in $L^\alpha(\Omega)$ and $\tilde g_k \to g$ strongly in $L^\beta(\Omega)$. 

\smallbreak

\noindent \textit{Third step.} $D_{p, q}$ is attained in $(f, g)$. Indeed, by Lemma \ref{lemma:regularity}, we have that also $K \tilde g_k \to K g$ strongly in $L^{p+1}$, thus
\[ \int_\Omega \tilde f_k K \tilde g_k \to \int_\Omega f K g =D_{p, q}. \]
Notice that $(f, g) \in X$, and, by passing to the limit in \eqref{cond}, we also have
\[ \gamma_1 \| f \|_\alpha^\alpha + \gamma_2 \norm{g}_\beta^\beta=1, \]
which concludes the proof. 
\end{proof}

\section{Proof of Theorem \ref{main thm}} 
\label{sec:proof}
Throughout this Section we assume that $q\leq p$  satisfy the assumptions of Theorem \ref{main thm} which, in this case, are equivalent to
\begin{equation}\label{eq:assumptions_pq}
\text{$p,q$ satisfy \eqref{CH}},\qquad  N\geq 6 \qquad \text{ and } \qquad \frac{N+2}{2(N-2)}< q\leq \frac{N+2}{N-2}
\end{equation}
or 
\begin{equation}\label{eq:assumptions_pq2}
\text{$p,q$ satisfy \eqref{CH}},\qquad  N= 5 \qquad \text{ and } \qquad \frac{17}{13}< q\leq \frac{7}{3}
\end{equation}
or
\begin{equation}\label{eq:assumptions_pq3}
\text{$p,q$ satisfy \eqref{CH}},\qquad  N= 4 \qquad \text{ and } \qquad \frac{7}{3}< q\leq 3.
\end{equation}

Recall from Subsection \ref{sec:bubbles} that we denote by $(U,V)\in W^{2,\beta}(\R^N)\times W^{2,\alpha}(\R^N)$ the radially decreasing solution of \eqref{eq:limitingsystem}, which satisfies the identity \eqref{eq:Spq_achieved} and the decay estimates \eqref{eq:decayestimate}. 

From now on, we choose $x_0$ to be a point on the boundary $\partial \Omega$ such that the curvature is positive. Without loss of generality, up to a rotation and a translation, we assume that
\[
x_0=0\in \partial \Omega,
\]
and that for a small $\eta>0$ we have
\begin{equation}\label{eq:boundary close to 0}
\Omega \cap B_\eta(0)=\{(x',x_N):\ x_N>\rho(x')\},\qquad \rho(x')=\sum_{j=1}^{N-1} \rho_j x_j^2 + O(|x'|^3).
\end{equation}
Therefore, for $x\in B_\eta(0)\cap \partial \Omega$,
\begin{equation}\label{boundary}
x_N=\sum_{j=1}^{N-1} \rho_j x_j^2 + O(|x'|^3),\quad \text{ and the mean curvature at $0$ is } \quad H(0)=\frac{2\sum_{j=1}^{N-1} \rho_j}{N-1}>0.
\end{equation} 
Here and henceforth we denote $x'=(x_1, \dots, x_{N-1})$. 
An outward normal close to $0$ has the expression
\[
(\nabla \rho(x'),-1)=(2\rho_1 x_1 + O(\abs{x'}^2), \dots, 2\rho_{N-1} x_{N-1} + O(\abs{x'}^2), -1),
\]
Hence the unitary exterior normal has the form:
\begin{equation}\label{eq:unitarynormal}
\nu(x)=\frac{(2\rho_1 x_1 + O(\abs{x'}^2), \dots, 2\rho_{N-1} x_{N-1} + O(\abs{x'}^2), -1)}{\sqrt{1+O(|x'|^2)}} \qquad \text{ as } x'\to 0.
\end{equation}

Recalling \eqref{eq:familybubbles},  we take, for $\varepsilon>0$, the scalings of $(U,V)$ centered at $x_0=0$:
\[U_\varepsilon(x):= U_{\varepsilon, 0} (x)=
\varepsilon^{-\frac N{p+1}} U \left(\frac{x}{\varepsilon} \right), \qquad V_\varepsilon(x):=V_{\varepsilon, 0} (x)=\varepsilon^{-\frac N{q+1}} V \left(\frac{x}{\varepsilon} \right), \]
which belong to $W^{2,\beta}(\Omega)\times W^{2,\alpha}(\Omega)$ and solve
\[
-\Delta U_{\varepsilon} = V_{\varepsilon}^q, \qquad -\Delta V_{\varepsilon} = U_{\varepsilon}^p\qquad  \text{ in  } \R^N.
\]

Let us define
\[ \tildeue=\ue^p +\we, \qquad \tildeve=\ve^q+\ze, \]
where 
\[ \we:=-\frac{1}{|\Omega|} \int_\Omega \ue^p, \qquad \ze:=-\frac{1}{|\Omega|} \int_\Omega \ve^q. \]
This way, $(\tildeue, \tildeve) \in X$ and we can, in particular, apply the operator $K$ to each component.

We prove the following result below.
\begin{proposition}\label{claim}
Assume $p,q$ satisfy \eqref{eq:assumptions_pq}, \eqref{eq:assumptions_pq2} or \eqref{eq:assumptions_pq3}.
%\[ q > \frac{N+4}{2(N-2)}. \]
Then, there exists a positive constant $c$ such that one has 
\[ \frac{\displaystyle \int_\Omega \tildeue K \tildeve}{\| \tildeue \|_\alpha \| \tildeve \|_\beta}  \ge \frac{2^{\frac2N}}{S_{p,q}} + c \eps + o(\eps)\qquad \text{ as } \eps\to 0^+. \]
\end{proposition} 
Notice that, once Proposition \ref{claim} is proved, then we may conclude the following.

\begin{proof}[Proof of Theorem \ref{main thm}]
Under the assumptions of Theorem \ref{main thm}, assume $q\leq p$. Then \eqref{eq:assumptions_pq}, \eqref{eq:assumptions_pq2} or \eqref{eq:assumptions_pq3} holds true, and from \eqref{def D eq 2} and Proposition  \ref{claim}  we have
\[ 
D_{p,q}=\mathop{\sup_{(f,g)\in X}}_{f,g\neq 0}\frac{\int_\Omega f Kg}{\|f\|_\alpha \|g\|_\beta}\geq  \frac{\int_\Omega \tildeue K \tildeve}{\| \tildeue \|_\alpha \| \tildeve \|_\beta} > \frac{2^{\frac2N}}{S_{p,q}} + \frac{c}{2} \eps > \frac{2^{2/N}}{S_{p, q}} \] 
by taking $\eps$ small enough. The same is true in case $p\leq q$, by exchanging the roles of $p$ and $q$, and $U$ and $V$. In conclusion, Theorem \ref{main thm} is proved  by recalling Lemma \ref{condition D}. Regularity follows from Proposition \ref{reg2}. 
\end{proof}

In order to prove Proposition \ref{claim}, we generalize to systems the ideas from \cite{CK} and use estimates in \cite{HMV, HV}. 
We preliminary define $\Phi_\eps$ and $\Psi_\eps$ in the following way:
\[ K \tildeue=\ve + \psie, \qquad K \tildeve=\ue+\phie, \]
so that, in particular,
\[ \int_\Omega(\ue+\phie)=0, \qquad \partial_\nu (\ue+\phie)=0 \text { on } \partial \Omega, \]
and
\[ \int_\Omega(\ve+\psie)=0, \qquad \partial_\nu (\ve+\psie)=0 \text{ on } \partial \Omega.\]

Equivalently,  $\Phi_\eps$ and $\Psi_\eps$ satisfy 
\[ 
\begin{cases}
-\Delta \phie = \ze &\text{ in } \Omega \\
\partial_\nu \phie= - \partial_\nu \ue & \text{ on } \partial \Omega \\
\displaystyle \int_\Omega \phie = - \int_\Omega \ue & 
\end{cases}\qquad \text{ and } \qquad \begin{cases}
-\Delta \psie = \we &\text{ in } \Omega \\
\partial_\nu \psie= - \partial_\nu \ve & \text{ on } \partial \Omega \\
\displaystyle \int_\Omega \psie = - \int_\Omega \ve. & 
\end{cases} \]

The proof of Proposition \ref{claim} is split into several lemmas. 
\begin{lemma}\label{step1}
 One has
\[  \int_\Omega \tildeue K \tildeve =\int_\Omega \ue^{p+1} + \int_\Omega \ve^{q+1} -\int_\Omega \nabla \ue \cdot \nabla \ve + \int_\Omega \nabla \psie \cdot \nabla \phie +  \int_\Omega \ve \ze + \int_\Omega \we \ue.\]
\end{lemma}
\begin{proof}
Using the definitions of $\tildeue$ and $\phie$, one gets 
\[ \int_\Omega \tildeue K \tildeve = \int_\Omega \tildeue(\ue+\phie)= \int_\Omega \ue^{p+1} + \int_\Omega \ue^p \phie + \int_\Omega \we \ue +\int_\Omega \we \phie. \]

Also,  by using the definition of the operator $K$,
\begin{align*}
\int_\Omega \ue^p \phie +  \int_\Omega \we \phie &= \int_\Omega (\ue^p + \we)\phie =\int_\Omega \nabla K(\ue^p + \we) \cdot \nabla \phie  \\
&= \int_\Omega \nabla (K \tildeue - \ve + \ve) \cdot \nabla \phie =  \int_\Omega\nabla \psie \cdot \nabla \phie + \int_\Omega \nabla \ve \cdot \nabla \phie  \\
&=  \int_\Omega\nabla \psie \cdot \nabla \phie + \int_\Omega \nabla \ve \cdot \nabla K \tildeve -\int_\Omega \nabla \ue \cdot \nabla \ve \\
&=  \int_\Omega \nabla \psie \cdot \nabla \phie + \int_\Omega \ve^{q+1}+  \int_\Omega \ve \ze -\int_\Omega \nabla \ue \cdot \nabla \ve,
\end{align*}
which yields the conclusion. 
\end{proof}

\begin{lemma}\label{step2} 
Assume $p,q$ satisfy \eqref{eq:assumptions_pq}, \eqref{eq:assumptions_pq2} or \eqref{eq:assumptions_pq3}. Then the following estimates hold
\[ \int_\Omega \ve \ze, \int_\Omega \we \ue = o(\eps) \qquad \text{ as } \eps\to 0^+. \]
\end{lemma}
\begin{proof}
Notice that, using the estimates  \eqref{eq:norm_estimates},
\[ 
\abs{\int_\Omega \ve \ze } \le \norm{\ze}_\infty \norm{\ve}_1 = \frac{1}{|\Omega|} \norm{\ve}_1\norm{\ve^q}_1 \approx
\begin{cases}
\eps^\frac{N}{p+1}\eps^\frac{N}{q+1}=\eps^{N-2}  & \text{ if } \frac{N}{N-2} < q \le \frac{N+2}{N-2}\\
\eps^{\frac{(N-2)^2}{2(N-1)}}\eps^\frac{N(N-2)}{2(N-1)}=\varepsilon^{N-2} |\log \varepsilon| & \text{ if }  q = \frac{N}{N-2}\\
\eps^\frac{N}{p+1}\eps^\frac{qN}{p+1}=\varepsilon^{\frac{N(q+1)}{p+1}} & \text{ if } q < \frac{N}{N-2}.
\end{cases}
\]
Observe that, since $N\geq 4$ and $q> \frac{N+2}{2(N-2)} \ge \frac{3}{N-2}$, in all three cases the quantities are $o(\eps)$.
For
\[ \abs{\int_\Omega \ue \we } \le \frac{1}{|\Omega|} \norm{\ue}_1\norm{\ue^p}_1,\]
using again  \eqref{eq:norm_estimates} we obtain the exact same decays, hence the proof is complete. 
\end{proof}

\begin{lemma}\label{step3} 
Assume $p,q$ satisfy \eqref{eq:assumptions_pq}, \eqref{eq:assumptions_pq2} or \eqref{eq:assumptions_pq3}. One has 
\[  \int_\Omega \nabla  \phie \cdot\nabla \psie =o(\eps)\qquad \text{ as } \eps\to 0^+. \]
\end{lemma} 
\begin{proof}
We first notice that
\[ \abs{ \int_\Omega \nabla \phie \cdot \nabla \psie} \le \left( \int_\Omega \abs{\nabla \phie}^2 \right)^{1/2} \left( \int_\Omega \abs{\nabla \psie}^2 \right)^{1/2}. \]
Recall that 
\[ \left \{ u \in H^1(\Omega): \int_\Omega u=0 \right \} \hookrightarrow L^{\frac{2(N-1)}{N-2}}(\partial \Omega) \]
in the sense of traces, 
see for instance \cite[Theorem 5.36]{Adams} (Trace Inequality) combined with \cite[Theorem 8.11]{LiebLoss} (Poincar\'e's inequality).
Let us start with the case $N \ge 5$.

\noindent Step 1. Estimate of $\displaystyle \left(\int_\Omega \abs{\nabla \phie}^2\right)^2$.

By integrating by parts,
\begin{align*} \nonumber \int_\Omega \abs{\nabla \phie}^2  &= - \frac{1}{|\Omega|} \int_\Omega \ve^q \int_\Omega \phie + \int_{\partial \Omega} \partial_\nu \phie \left(\phie - \frac{1}{|\Omega|} \int_\Omega  \phie + \frac{1}{|\Omega|} \int_\Omega \phie \right) \\
\nonumber& \le - \frac{1}{|\Omega|} \int_\Omega \ve^q \int_\Omega \phie  + \frac{1}{|\Omega|}  \int_{\partial \Omega} \partial_\nu \phie \int_\Omega  \phie \\
\nonumber&\hspace{4cm} + \norm{\partial_\nu \phie}_{L^{\frac{2(N-1)}{N}}(\partial \Omega)} \norm{ \phie - \frac{1}{|\Omega|} \int_\Omega  \phie }_{L^{\frac{2(N-1)}{N-2}}(\partial \Omega)}   \\
\nonumber& = \frac{1}{|\Omega|} \int_\Omega \ve^q \int_\Omega \ue  +\frac{1}{|\Omega|} \int_{\partial \Omega} \partial_\nu \ue  \int_\Omega  \ue\\
\nonumber & \hspace{4cm} + \norm{\partial_\nu \ue}_{L^{\frac{2(N-1)}{N}}(\partial \Omega)} \norm{ \phie - \frac{1}{|\Omega|} \int_\Omega  \phie }_{L^{\frac{2(N-1)}{N-2}}(\partial \Omega)}  \\
& \label{stima grad phi}  \le \frac{1}{|\Omega|} \norm{ \ve^q}_1 \norm{ \ue}_1  + \frac{1}{|\Omega|}\norm{\ue}_1  \int_{\partial \Omega} \partial_\nu \ue  +
C\norm{\partial_\nu \ue}_{L^{\frac{2(N-1)}{N}}(\partial \Omega)} \left( \int_\Omega \abs{\nabla \phie }^2 \right)^{1/2} .
\end{align*}
We now point out that 
\[ \left| \int_{\partial \Omega} \partial_\nu U_\eps \right| \le c \norm{\partial_\nu \ue}_{L^{\frac{2(N-1)}{N}}(\partial \Omega)}. \]
This observation, combined with Lemma \ref{eq:estimate_normal_derivative} - which provides the estimates of $\| \partial_\nu U_\eps\|_{L^{\frac{2(N-1)}{N}}(\partial \Omega)}$ - and  \eqref{eq:norm_estimates}, gives
\[ \int_\Omega \abs{\nabla \phie}^2 \le 
\begin{cases}
h_q(\eps)\left( \left(\int_\Omega \abs{\nabla \phie}^2 \right)^{1/2} + \eps^{\frac{N}{q+1}} \right) + \eps^{\frac{2N}{q+1}} &\text{ if } q > \frac{N}{N-2} \\
h_q(\eps) \left( \left(\int_\Omega \abs{\nabla \phie}^2 \right)^{1/2} +  \eps^{\frac{N(N-2)}{2(N-1)}} |\log \eps| \right) + \eps^{\frac{N(N-2)}{N-1}} |\log \eps|^2 &\text{ if } q = \frac{N}{N-2} \\
h_q(\eps) \left( \left(\int_\Omega \abs{\nabla \phie}^2 \right)^{1/2} + \eps^{\frac{Nq}{p+1}} \right) + \eps^{\frac{2qN}{p+1}} &\text{ if } q < \frac{N}{N-2},
\end{cases} \]
where $h_q(\eps)$ is defined in \eqref{stima norma}. 

We claim that
\begin{equation}\label{stima phi 1_resumo}
\left( \int_\Omega \abs{\nabla \phie}^2 \right)^{1/2} =
\begin{cases}
O(\eps^{-\frac{N}{p+1} + \frac N2}) & \text{ if } q>\frac{N+4}{2(N-2)}\\
O(\eps^{-\frac{N}{p+1} + \frac N2} |\log \eps|^{\frac{N}{2(N-1)}}) &\text{ if } q>\frac{N+4}{2(N-2)}\\
O(\eps^{-\frac{N}{p+1} + q(N-2) -2 }) & \text{ if } q<\frac{N+4}{2(N-2)}
\end{cases}.\end{equation}
Indeed:
\begin{itemize}
    \item If $q > \frac{N}{N-2}$, then
\[ \int_\Omega \abs{\nabla \phie}^2 \le C\, \left( \eps^{-\frac{N}{p+1} + \frac N2} \left(\int_\Omega \abs{\nabla \phie}^2 \right)^{1/2}  + \eps^{-\frac{N}{p+1} +\frac N2 +\frac{N}{q+1}} + \eps^{\frac{2N}{q+1}} \right), \]
thus, since $-\frac{N}{p+1} + \frac N2<-\frac{N}{p+1} +\frac N2 +\frac{N}{q+1},\frac{2N}{q+1}$,
\begin{equation}\label{stima phi 1} \left( \int_\Omega \abs{\nabla \phie}^2 \right)^{1/2} \le C\, \eps^{-\frac{N}{p+1} + \frac N2}. \end{equation}
\item If $q=\frac{N}{N-2}(>\frac{N+4}{2(N-2)})$, then computations are analogous since we use the same estimate as before from \eqref{stima norma}, and we still  get  \eqref{stima phi 1}. 
\item For the case $\frac{N+4}{2(N-2)} < q < \frac{N}{N-2}$, we preliminary notice that
\[ \frac{q+1}{p+1} > \frac1 2 \quad \text{ if and only if } \quad q> \frac{N+4}{2(N-2)}. \]
Then
\[ \int_\Omega \abs{\nabla \phie}^2 \le C\, \left( \eps^{-\frac{N}{p+1} + \frac N2} \left(\int_\Omega \abs{\nabla \phie}^2 \right)^{1/2}  + \eps^{N \frac{q-1}{p+1} + \frac N2} + \eps^{\frac{2qN}{p+1}}\right), \]
and again \eqref{stima phi 1} holds.
\item Let now $q=\frac{N+4}{2(N-2)}$. Then 
\[ \int_\Omega \abs{\nabla \phie}^2 \le C\, \left( \eps^{-\frac{N}{p+1} + \frac N2}  |\log \eps|^{\frac{N}{2(N-1)}} \left(\int_\Omega \abs{\nabla \phie}^2 \right)^{1/2}  + \eps^{N \frac{q-1}{p+1} + \frac N2}  |\log \eps|^{\frac{N}{2(N-1)}}  + \eps^{\frac{2qN}{p+1}} \right), \]
which gives
\begin{equation*}\label{stima phi 2} \left( \int_\Omega \abs{\nabla \phie}^2 \right)^{1/2} \le C\, \eps^{-\frac{N}{p+1} + \frac N2} |\log \eps|^{\frac{N}{2(N-1)}}. \end{equation*}
\item Finally, let us consider $q < \frac{N+4}{2(N-2)}$. In this case, 
\[ \int_\Omega \abs{\nabla \phie}^2 \le C\, \left( \eps^{-\frac{N}{p+1} + q(N-2) -2 } \left(\int_\Omega \abs{\nabla \phie}^2 \right)^{1/2}  + \eps^{-\frac{N}{p+1} + q(N-2) -2 + \frac{Nq}{p+1}}   + \eps^{\frac{2qN}{p+1}}\right), \]
thus
\begin{equation*}\label{stima phi 3} \left( \int_\Omega \abs{\nabla \phie}^2 \right)^{1/2} \le C\,  \eps^{-\frac{N}{p+1} + q(N-2) -2 }. \end{equation*}
\end{itemize}
In conclusion, claim \eqref{stima phi 1_resumo} holds true.

\smallbreak

 \noindent Step 2. Estimate of $\displaystyle \left(\int_\Omega \abs{\nabla \psie}^2\right)^\frac{1}{2}$.

 Following the same proof above, we have
\[\int_\Omega |\nabla \Psi_\eps|^2 \le \frac{1}{|\Omega|} \norm{ \ue^p}_1 \norm{ \ve}_1  + \frac{1}{|\Omega|}\norm{\ve}_1  \int_{\partial \Omega} \partial_\nu \ve  +
C\norm{\partial_\nu \ve}_{L^{\frac{2(N-1)}{N}}(\partial \Omega)} \left( \int_\Omega \abs{\nabla \psie }^2 \right)^{1/2} .
\]
By Lemma \ref{eq:estimate_normal_derivative2},
\[ \int_\Omega \abs{\nabla \psie}^2 \le 
c\eps^{-\frac{N}{q+1}+\frac{N}{2}}\left( \left(\int_\Omega \abs{\nabla \psie}^2 \right)^{1/2} + \eps^{\frac{N}{p+1}} \right) + \eps^{\frac{2N}{p+1}}, \]
hence
\[ \left( \int_\Omega \abs{\nabla \psie}^2 \right)^{1/2} \le
\eps^{-\frac{N}{q+1} + \frac N2}. \]

\smallbreak

Therefore, by \eqref{stima phi 1} and combining the previous two steps, 
\[
\int_\Omega \nabla \phie \cdot \nabla \psie \le 
\begin{cases}
c\, \eps^{-\frac{N}{q+1}-\frac{N}{p+1}+N}=c\, \eps^{2} &\text{ if } q > \frac{N+4}{2(N-2)} \\
c\, \eps^2 |\log \eps|^{\frac{N}{2(N-1)}}   &\text{ if } q = \frac{N+4}{2(N-2)} \\
c\, \eps^{q(N-2)-\frac{N}{2}} &\text{ if } q < \frac{N+4}{2(N-2)}.
\end{cases}
\]
It is immediate that this is an $o(\eps)$ if $q > \frac{N+2}{2(N-2)}$.  

\smallbreak

We now consider the case $N=4$. 
Here, the same integration by parts as in Step 1, combined with estimates \eqref{stima norma N=4} and \eqref{eq:norm_estimates}, provides
\[ \int_\Omega \abs{\nabla \phie}^2 \le c \left(  \eps^{2 \frac{p-1}{p+1}} (\log \eps)^{\frac 23} \left( \int_\Omega \abs{\nabla \phie}^2  \right)^{\frac 12} + \eps^{4\frac{p-1}{p+1}} (\log \eps)^{\frac 23 } +  \eps^{4\frac{p-1}{p+1}} \right) \] 
from which 
\[ \left( \int_\Omega \abs{\nabla \phie}^2  \right)^{\frac 12} \le c \eps^{2\frac{p-1}{p+1}} (\log \eps)^{\frac 23 }. \]
Analogously, using Lemma \ref{eq:estimate_normal_derivative2}, 
\[ \left( \int_\Omega \abs{\nabla \psie}^2  \right)^{\frac 12} \le c \eps^{2\frac{q-1}{q+1}} (\log \eps)^{\frac 23 }. \]
Therefore,
\[ \int_\Omega \nabla \Phi_\eps \cdot \nabla \Psi_\eps \le c \eps^2 (\log \eps)^{\frac 43 } =o(\eps), \]
from which the conclusion. 
\end{proof} 

\begin{lemma}\label{step4} 
Assume $p,q$ satisfy \eqref{eq:assumptions_pq}, \eqref{eq:assumptions_pq2} or \eqref{eq:assumptions_pq3}. There exists a small positive value $\eta$ such that  
\[  \int_{\partial \Omega} \ue \partial_\nu \ve = \int_{\partial \Omega \cap B_\eta(0)} \ue \partial_\nu \ve + o(\eps), \quad \text{ and } \quad \int_{\partial \Omega \cap B_\eta(0)} \ue \partial_\nu \ve <0. \]
Analogously, one gets, for a suitable small $\eta$,
\[  \int_{\partial \Omega} \ve \partial_\nu\ue = \int_{\partial \Omega \cap B_\eta} \ve \partial_\nu \ue + o(\eps), \quad \text{ and } \int_{\partial \Omega \cap B_\eta(0)} \ve \partial_\nu \ue <0.\]
\end{lemma}
\begin{proof}
Let us preliminary notice that $rV'(r) \approx r^{2-N}$ as $r\to \infty$,
and
\begin{equation}\label{normal derivative}
\partial_\nu V_\eps(x) = \eps^{-\frac{N}{q+1} - 1}  V'(|x'|/\eps) \frac{ \sum_{j=1}^{N-1} \rho_j x_j^2 + O(|x'|^3)}{|x|\sqrt{1+|x'|^2}}. 
\end{equation}
Therefore,  taking into account \eqref{eq:norm_estimates},
\[ 
\norm{\partial_\nu \ve}_{L^\infty(\partial \Omega \setminus B_\eta)} \approx \eps^{-\frac{N}{q+1}+N-2}= \eps^{\frac{N}{p+1}}. \]
Recall also from \cite[pp. 2360--2362]{HMV} that 
\[ \norm{U_\eps}_{L^\infty(\partial \Omega \setminus B_\eta)} \approx \|U_\eps\|_1 \approx
\begin{cases}
\eps^{\frac{N}{q+1}} & \text{ if } q > \frac{N}{N-2} \\
\eps^{\frac{N(N-2)}{2(N-1)}} |\log \eps|  & \text{ if } q = \frac{N}{N-2}  \\
 \eps^{\frac{qN}{p+1}}  & \text{ if } q < \frac{N}{N-2}. 
\end{cases} 
\]

Thus 
\begin{align*}
\left|\int_{\partial \Omega \setminus B_\eta} \ue \partial_\nu \ve \right| &\le c \norm{\ue}_{L^\infty(\partial \Omega \setminus B_\eta)}\norm{\partial_\nu \ve}_{L^\infty(\partial \Omega \setminus B_\eta)}=\left.\begin{cases}
O(\eps^{N-2})& \text{ if } q > \frac{N}{N-2}\\
O(\eps^{N-2}\log \eps)& \text{ if } q = \frac{N}{N-2}\\
O(\eps^{q(N-2)-2})& \text{ if } q < \frac{N}{N-2}
\end{cases}\right\}=o(\eps).  \end{align*}
where in the last identity we use the fact that $N\geq 4$ and $q >\frac{N+2}{2(N-2)} \ge \frac{3}{N-2}$.
Moreover, \eqref{normal derivative} shows that 
\[  \int_{\partial \Omega \cap B_\eta} \ue \partial_\nu \ve <0, \]
if $\eta$ is small enough;
indeed, $V'<0$ and $0$ is a point of positive curvature and $\sum_{j=1}^{N-1} \rho_j x_j^2+O(|x'|^3)>0$ as $|x'|\to 0$ (recall \eqref{eq:boundary close to 0}--\eqref{boundary}). The lemma is proved. 
\end{proof}

\begin{lemma}\label{step5} 
Assume $p,q$ satisfy \eqref{eq:assumptions_pq}, \eqref{eq:assumptions_pq2} or \eqref{eq:assumptions_pq3}. As $\eps\to 0^+$, one has
\[ \| \tildeue \|_{\alpha} \le \norm{\ue}_{p+1}^p + o(\eps), \qquad \| \tildeve \|_{\beta} \le \norm{\ve}_{q+1}^q + o(\eps). \]
\end{lemma}
\begin{proof}
We preliminary notice that, exploiting estimates in \eqref{eq:norm_estimates},
\begin{align*} 
\int_\Omega \ue |\we| &=\frac{1}{|\Omega|}\|U_\eps\|_1\|U_\eps^p\|_1=\left.\begin{cases}
O(\eps^{N-2})& \text{ if } q > \frac{N}{N-2}\\
O(\eps^{N-2}\log \eps)& \text{ if } q = \frac{N}{N-2}\\
O(\eps^{q(N-2)-2})& \text{ if } q < \frac{N}{N-2}
\end{cases}\right\}\, =o(\eps).
\end{align*}
By exploiting the inequality:
\[
|| a+\left.b\right|^\alpha-|a|^\alpha \mid \leq C\left(|a|^{\alpha-1}|b|+|b|^\alpha\right) \quad \forall a, b \in \mathbb{R}
\]
(recall that $\alpha=\frac{p+1}{p}>1$, so this is a consequence of a Taylor expansion with Lagrange remainder), and once again \eqref{eq:norm_estimates}, we obtain
\begin{align*} 
\| \tildeue \|_{\alpha} ^\alpha &=\int_\Omega(U_\eps^p+W_\eps)^\alpha \leq \int_\Omega U_\eps^{p+1}+C\int_\Omega (|W_\eps|^\alpha+U_\eps|W_\eps|) \le C\left( \norm{\ue}_{p+1}^{p+1} + \norm{\ue^p}_1^\alpha + \int_\Omega \ue |\we| \right)\\
& \le C\left( \norm{\ue}_{p+1}^{p+1} + \norm{\ue^p}_1^\alpha  + o(\eps) \right)\le C \left( \norm{\ue}_{p+1}^{p+1} + O(\eps^\frac{N}{p})  + o(\eps) \right) \le C\norm{\ue}_{p+1}^{p+1} + o(\eps).  \end{align*}
Notice that, due to \eqref{eq:assumptions_pq}, \eqref{eq:assumptions_pq2} and \eqref{eq:assumptions_pq3}, 
one has $p < N$, hence $O(\eps^\frac{N}{p})=o(\eps)$. 
The conclusion follows using the Taylor expansion
\begin{equation}\label{Taylor}
    (y+x)^s=y^s +s y^{s-1} x  +y^s o\left(xy^{-1} \right) \qquad \text{ as } xy^{-1}\to 0, \end{equation}
with $s=1/\alpha$, $y=\norm{U_\eps}^{p+1}_{p+1}$ and $x=o(\eps)$.
Notice that 
\[ \norm{U_\eps}^{p+1}_{p+1}=\int_{\Omega/\eps} U(y)^{p+1} \, dy 
\begin{cases}
\ge \int_{\frac 1 \eps \left( B_\eta(0) \cap \Omega \right)} U(y)^{p+1} \, dy \to \int_{\R^N_+} U(y)^{p+1} \, dy = \frac 12 S^{\frac N2} \\
\leq \int_{\R^N_+} U(y)^{p+1} \, dy = \frac 12 S^{\frac N2}
\end{cases}
\quad \text{ as } \eps \to 0,
\]
hence it is bounded from above, and below away from zero.

Similarly we can treat $\tildeve$. 
This time we get
\[ \| \tildeve \|_{\beta}^\beta \le 
\begin{cases}
C \|\ve\|_{q+1}^{q+1} + O(\eps^{N/q}) + o(\eps) &\text{ if $q > \frac{N}{N-2}$} \\
C \|\ve\|_{q+1}^{q+1} + O(\eps^{\frac{N-2}{2}}|\log \eps|^{\frac{2(N-1)}{N}}) + o(\eps) &\text{ if $q = \frac{N}{N-2}$} \\
C \|\ve\|_{q+1}^{q+1} + O(\eps^{\frac{N(q+1)}{p+1}}) + o(\eps) &\text{ if $q < \frac{N}{N-2}$.}
\end{cases} \]
Notice that in case $N=4$, then \eqref{eq:assumptions_pq3} implies $q > 2=\frac{N}{N-2}$, hence we are in the first case, and using $q<N$, 
\begin{equation}\label{stima:211} \| \tildeve \|_{\beta}^\beta \le C \|\ve\|_{q+1}^{q+1} +o(\eps). \end{equation}
If $N \ge 5$, any of the three cases above can occur, however $N > q$, $\frac{N-2}{2} > 1$ and $\frac{N(q+1)}{p+1}>1$, hence 
\eqref{stima:211} still holds. 
\end{proof}

\begin{lemma}\label{step6} 
Assume $p,q$ satisfy \eqref{eq:assumptions_pq}, \eqref{eq:assumptions_pq2} or \eqref{eq:assumptions_pq3}.
One has, as $\eps\to 0^+$,
\[ \norm{\ue}_{p+1}^{p+1} \le \frac12 S_{p,q}^{\frac N2} -C_1 \eps  + o(\eps), \qquad  \norm{\ve}_{q+1}^{q+1} \le \frac12 S_{p,q}^{\frac N2}  -C_2 \eps  + o(\eps), \]
for the constants
\[
C_1:=\frac{H(0)}{2}\int_{\R^{N-1}} |y'|^2U^{p+1}(y',0)\, dy'>0,\quad C_2:=\frac{H(0)}{2}\int_{\R^{N-1}} |y'|^2V^{q+1}(y',0)\, dy'>0,
\]
where $H(0)$ is the mean curvature at $0$ (recall \eqref{boundary}).
\end{lemma} 
\begin{proof}
Take $\eta$ as in \eqref{eq:boundary close to 0}, and recall that we assume that $0\in \partial \Omega$ has positive mean curvature. Let us first observe that 
\begin{align*}
\int_\Omega \ue^{p+1} &= \int_{\Omega \cap B_\eta(0)} \ue^{p+1} + \int_{\Omega \setminus B_\eta(0)} \ue^{p+1} = \frac 12 \int_{B_\eta(0)} \ue^{p+1} - \int_\Sigma \ue^{p+1} + \int_{\Omega \setminus B_\eta(0)} \ue^{p+1} \\
 & \le  \frac 12 \int_{\R^N}  \ue^{p+1} - \int_\Sigma \ue^{p+1}  + \int_{\R^N \setminus B_\eta(0)} \ue^{p+1},
\end{align*}
where
\begin{align*} \Sigma:= B_\eta^+(0) \cap \Omega^c&=\{x\in B_\eta(0)\cap \Omega^c: x_n>0\}=\{x\in B_\eta(0):\ 0<x_n<\sum_{j=1}^{N-1} \rho_j x_j^2 + O(|x'|^3)\}.
\end{align*}
Now, recalling Section \ref{sec:bubbles},
\begin{equation}\label{stima1} \frac 12 \int_{\R^N}  \ue^{p+1}= \frac 12 S_{p,q}^{\frac N2}. \end{equation}
Also, in the case $q > \frac{N}{N-2}$, by the decay estimates \eqref{eq:decayestimate},
\begin{align} 
\nonumber \int_{\R^N \setminus B_\eta} \ue^{p+1} &= \eps^{-N} \int_{B_\eta^c} U(|x|/\eps)^{p+1} = O(\eps^{-N+(p+1)(N-2)}) \int_\eta^{+\infty} r^{-(p+1)(N-2)+N-1} \\
 &= O(\eps^{-N+(p+1)(N-2)})=o(\eps)  \label{stima2} \end{align} 
since $p>\frac{3}{N-2}$. Similarly in the other cases. 

We now estimate the integral on $\Sigma$. We first notice that 
\begin{align*} \int_\Sigma \ue^{p+1}  &=\eps^{-N} \int_\Sigma U(x/\eps)^{p+1} = \eps^{-N}\int_{\Delta_\eta} \int_0^{\sum_j \rho_j x_j^2 + O(|x'|^3)} U(x/\eps)^{p+1} dx_N \, dx'\\
&=\int_{\Delta_\eta/\eps} \int_0^{\eps \sum_j \rho_j y_j^2 + \eps^2 O(|y'|^3)} U(y)^{p+1} dy_N \, dy'.\end{align*}
Now we claim that 
\begin{equation}\label{claim expansion} \int_0^s U(|y|)^{p+1} dy_N  = U^{p+1}(y', 0)(s+ O(s^2)). \end{equation}
Indeed, consider 
\[ f(s)=\int_0^s U^{p+1}(y) dy_N \]
and Taylor expand it. Notice that 
\[ 
f'(0)=U^{p+1}(y',0),  \qquad f''(s)= (p+1) U^p(y', s) \frac{\partial U}{\partial y_N} (y', s). \]
Also, 
 \begin{equation}\label{bound second derivative} |f''(s)| \le CU^{p+1}(y', 0) \end{equation}
where $C$ is a positive constant independent of $y', s$: to get this bound, one observes that $U$ has a maximum at $0$, and 
\[ \abs{\frac{\partial U}{\partial y_N}(y)} \le \abs{\frac{\partial U}{\partial y_N}(y', 0)} \le C U(y', 0) \]
if $|y|>R$ with $R$  big enough, by decay estimates in \eqref{decay derivative}, and
\[ \abs{\frac{\partial U}{\partial y_N}(y)} \le C U(y', 0) \]
if $|y| \le R$, for some positive constant $C=C(R)>0$, since they are all bounded quantities. Hence \eqref{bound second derivative} holds, and the claim \eqref{claim expansion} is proved. 

Therefore, 
\[ \int_\Sigma \ue^{p+1}  = \int_{\Delta_\eta/\eps} U^{p+1}(y', 0)  \left[ ( \eps \sum_j \rho_j y_j^2 + \eps^2 O(|y'|^3) ) + O( \eps \sum_j \rho_j y_j^2 + \eps^2 O(|y'|^3) )^2\right].
\]
Notice that 
\[\int_{\R^{N-1}} |y'|^{3} U^{p+1}(y', 0)\, dy',\  \int_{\R^{N-1}} |y'|^6 U^{p+1}(y', 0)\, dy' < +\infty, \]
due to the decay estimates for $U$, \eqref{eq:decayestimate}, 
thus 
\begin{align} \nonumber \int_\Sigma \ue^{p+1}  &= \eps \int_{\R^{N-1}} \sum_j \rho_j y_j^2 \, U^{p+1}(y', 0) + o(\eps) \\
\nonumber &= \eps \sum_j \rho_j \int_{\R^{N-1}} y_1^2 \, U^{p+1}(y', 0) + o(\eps) \\
\label{stima3} &= C_1\eps + o(\eps).
\end{align}
The conclusion follows due to \eqref{stima1}, \eqref{stima2} and \eqref{stima3}. The proof for $V_\eps$ is analogous.
\end{proof}

We are now ready to prove Proposition \ref{claim}.
\begin{proof}[Proof of Proposition \ref{claim}] 
Since for any $0 < \lambda <1$ 
\begin{align*} \int_\Omega \nabla \ue \cdot \nabla \ve 
&=\lambda \int_\Omega \nabla \ue \cdot \nabla \ve + (1-\lambda) \int_\Omega \nabla \ue \cdot \nabla \ve \\
&= \lambda \int_\Omega \ue^{p+1}  + (1-\lambda)\int_\Omega \ve^{q+1} +\lambda \int_{\partial \Omega} \ue \partial_\nu \ve +(1-\lambda) \int_{\partial \Omega} \ve \partial_\nu \ue, 
\end{align*}
and also recalling Lemmas \ref{step1}, \ref{step2} and \ref{step3}, we have 
\[ \int_\Omega \tildeue K \tildeve = (1-\lambda) \int_\Omega \ue^{p+1}  + \lambda \int_\Omega \ve^{q+1} - \lambda \int_{\partial \Omega} \ue \partial_\nu \ve -(1-\lambda) \int_{\partial \Omega} \ve \partial_\nu \ue +o(\eps).\]
Therefore, using Lemma \ref{step4} and Young's inequality, 
\begin{align*} \int_\Omega \tildeue K \tildeve& \ge (1-\lambda) \int_\Omega \ue^{p+1}  + \lambda \int_\Omega \ve^{q+1} +o(\eps) \ge \norm{\ue}_{p+1}^{(1-\lambda) (p+1)}  \norm{\ve}_{q+1}^{\lambda(q+1)} +o(\eps). \end{align*}
We now exploit Lemma \ref{step5} and \eqref{Taylor} to obtain 
\begin{align}\nonumber
\frac{\int_\Omega \tildeue K \tildeve}{\| \tildeue \|_\alpha \| \tildeve \|_\beta}  &\ge \frac{\norm{\ue}_{p+1}^{(1-\lambda) (p+1)}  \norm{\ve}_{q+1}^{\lambda(q+1)} +o(\eps)}{(\|U_\eps\|_{p+1}^p+o(\eps))(\|V_\eps\|_{p+1}^q+o(\eps))} \\
&\ge \norm{\ue}_{p+1}^{(p+1)(1-\lambda - \frac{p}{p+1})} \norm{\ve}_{q+1}^{(q+1)(\lambda - \frac{q}{q+1})}+ o(\eps).  \label{stima basso}
\end{align}

We choose $\lambda$ such that 
\[ \frac{1}{p+1} <\lambda < \frac{q}{q+1}, \]
which exists since $pq>1$.

By the Taylor expansion \eqref{Taylor}  and Lemma \ref{step6} we obtain, since $1-\lambda-\frac{p}{p+1}=\frac{1}{p+1}-\lambda<0$,
\begin{align*} \norm{\ue}_{p+1}^{(p+1)(1-\lambda - \frac{p}{p+1})} & \ge \left(\frac{1}{2}S_{p,q}^\frac{N}{2}-C_1\eps+o(\eps)\right)^{\frac{1}{p+1}-\lambda}\\
&\ge \left( \frac 12 S_{p,q}^{\frac N2} \right)^{\frac 1{p+1} -\lambda } -C_1 \left(\frac 1{p+1} -\lambda\right)\left( \frac 12 S_{p,q}^{\frac N2} \right)^{\frac 1{p+1} -\lambda -1}\eps + o(\eps) \\
& =  \left( \frac 12 S_{p,q}^{\frac N2} \right)^{\frac 1{p+1} -\lambda } + c_1 \eps + o(\eps)   \end{align*}
with $c_1 >0$ due to our choice of $\lambda$. Similarly, we obtain
\[\norm{\ve}_{q+1}^{(q+1)(\lambda - \frac{q}{q+1})}  \ge  \left( \frac 12 S_{p,q}^{\frac N2} \right)^{\lambda - \frac{q}{q+1}} + c_2 \eps + o(\eps)    \]
with $c_2>0$.

Finally, recalling \eqref{stima basso}, 
\begin{align*} \frac{\int_\Omega \tildeue K \tildeve}{\| \tildeue \|_\alpha \| \tildeve \|_\beta} & \ge \left( \frac12 S_{p,q}^{\frac N2} \right)^{-1+\frac{1}{p+1}+\frac{1}{q+1}} + c \eps + o(\eps)=  \frac{2^{\frac2N}}{S_{p,q}} + c \eps + o(\eps) > \frac{2^{\frac2N}}{S_{p,q}}, \end{align*}
for a positive constant $c>0$, if 
$\varepsilon$ is small enough.
\end{proof}

\section{A perturbation argument: proof of Theorem \ref{main thm1.2}. }\label{sec: close to 1}
In this section, we prove Theorem \ref{main thm1.2}. The idea is to study the case $(p,q)=(1,\frac{N+4}{N-4})$,  when system \eqref{system} reduces to the single equation \eqref{biharmonic}, proving Theorem \ref{main thm 2}, and then to consider a perturbation argument. 

\begin{lemma}\label{pq close}
Assume that, for some $p^*, q^*$ satisfying \eqref{CH}, the following two conditions hold:
\begin{itemize}
\item[(i)] $\displaystyle D_{p^*, q^*} > \frac{2^{2/N}}{S_{p^*, q^*}}$;

\item[(ii)] $S_{p, q} \to S_{p^*, q^*}$ as $(p,q)\to (p^*,q^*)$ on the critical hyperbola.
\end{itemize}
Then there exists $\eps=\eps(\Omega,N)>0$ such that
\[ 
D_{p, q} > \frac{2^{\frac 2N}}{S_{p, q}} 
\]
for $(p, q)$ satisfying \eqref{CH} with $|p-p^*|+|q- q^*|<\eps$. 
\end{lemma}
\begin{proof}
By (i) and Lemma \ref{condition D}, there exists $(f_*,g_*)\in X^{\frac{p^*+1}{p^*}}\times X^{\frac{q^*+1}{q^*}}$ which achieves $D_{p_*,q_*}$. For any $L>0$ we define 
\[ \tilde f_L(x)=
\begin{cases}
    f_*(x) &\text{ if } |f_*(x)| < L \\
    L & \text{ if } |f_*(x)| \ge L
\end{cases}\qquad \tilde g_L(x)=
\begin{cases}
    g_*(x) &\text{ if } |g_*(x)| < L \\
    L &\text{ if } |g_*(x)| \ge L
\end{cases}\]
and
\[ f_L(x)=\tilde f_L(x)- \frac{1}{|\Omega|} \int_\Omega \tilde f_L,\qquad g_L(x)=\tilde g_L(x)- \frac{1}{|\Omega|}  \int_\Omega \tilde g_L. \]
Notice that $f_L,g_L \in L^\infty(\Omega)$ and have zero average, thus $(f_L,g_L)\in X^\alpha\times X^\beta$ for every $(p,q)$  and $(f_L, g_L)$ is a test function for $D_{p, q}$. Moreover, 
\[ \| f_L \|_{\frac{p+1}{p}} \to  \| f_L \|_{\frac{p^*+1}{p^*}},\qquad  g_L \|_{\frac{q+1}{q}} \to  \| g_L \|_{\frac{q^*+1}{q^*}}\]
as $(p,q)\to (p^*,q^*)$ by exploiting the dominated convergence theorem. Thus, 
\begin{equation}\label{eq:aux_cont} \lim_{(p, q) \to (p^*, q^*)} D_{p, q} \ge \lim_{(p, q) \to (p^*, q^*)}  \frac{\int_\Omega f_L K g_L}{\|f_L\|_\alpha \|g_L\|_\beta} = \frac{\int_\Omega f_L K g_L}{\|f_L\|_{\frac{p^*+1}{p^*}} \|g_L\|_{\frac{q^*+1}{q^*}} }. \end{equation}

We claim that, as $L \to \infty$, 
\begin{equation}\label{claim conv fL}
    \|f_L\|_{\frac{p^*+1}{p^*}} \to \|f_*\|_{\frac{p^*+1}{p^*}} \quad \text{ and } \quad \|g_L\|_{\frac{q^*+1}{q^*}} \to \|g_*\|_{\frac{q^*+1}{q^*}} 
    \end{equation} 
and
\begin{equation}\label{claim conv fLKgL}
\int_\Omega f_L K g_L \to \int_\Omega f_* K g_*. \end{equation}
Indeed, 
\begin{equation}\label{conv lp}
\tilde f_L \to f_* \quad \text{ in } L^\frac{p^*+1}{p^*},\qquad  \tilde g_L \to g_* \quad \text{ in } L^\frac{q^*+1}{q^*}
\end{equation}
by dominated convergence, and so
\[ \int_\Omega \tilde f_L \to \int_\Omega f_*=0, \quad \int_\Omega \tilde g_L \to \int_\Omega g_*=0.
\]
From this we have that \eqref{claim conv fL} follows; by the continuity properties of the operator $K$, also \eqref{claim conv fLKgL} holds true.

Thus, using $(i)$-$(ii)$ and \eqref{eq:aux_cont},  
\begin{align*}
\lim_{(p, q) \to (p^*, q^*)} D_{p, q}&=
\lim_{L \to \infty} \lim_{(p, q) \to (p^*, q^*)} D_{p, q}  \ge \lim_{L \to \infty} \frac{\int_\Omega f_L K g_L}{\|f_L\|_{\frac{p^*+1}{p^*}} \|g_L\|_{\frac{q^*+1}{q^*}} } \\
&= \frac{\int_\Omega f_* K g_*}{\|f_*\|_{\frac{p^*+1}{p^*}} \|g_*\|_{\frac{q^*+1}{q^*}} } =D_{p^*, q^*} > \frac{2^{\frac 2N}}{S_{p^*, q^*}} = \lim_{(p, q) \to (p^*, q^*)}  \frac{2^{\frac 2N}}{S_{p, q}}. 
\end{align*}
Thus 
\[ D_{p, q} > \frac{2^{\frac 2N}}{S_{p, q}}\quad \text{for $(p, q)$ close enough to $(p^*, q^*)$. } \qedhere \]
\end{proof}

The following Lemma immediately implies Theorem \ref{main thm 2} and it will be the starting point of our perturbation argument.
\begin{lemma}\label{sigma attained}
Let 
\[ \Sigma =\inf_{u \in H^2_\nu(\Omega)} \left \{ \int_\Omega |\Delta u|^2 : \int_\Omega |u|^{\frac{2N}{N-4}} =1 \right \},\quad \text{ where }\quad
 H^2_\nu(\Omega)= \left \{ u \in H^2(\Omega): \partial_\nu u= 0 \text{ on } \partial \Omega \right\}. \] 
Then, $\Sigma$ is attained, and 
\begin{equation}\label{dis sigma} \Sigma < \frac{S^2_{1, \frac{N+4}{N-4}}}{2^{\frac 4N}}.\end{equation}
\end{lemma}
\begin{proof}
Notice that \eqref{dis sigma} is an easy consequence of \cite[Lemma 4.1]{BCN}, as for any $\alpha>0$ one has
\[ \Sigma \le \inf_{u \in H^2_\nu(\Omega)} \left \{ \int_\Omega (|\Delta u|^2 + |\nabla u|^2 + \alpha |u|^2): \int_\Omega |u|^{\frac{2N}{N-4}} =1 \right \} < \frac{S^2_{1, \frac{N+4}{N-4}}}{2^{\frac 4N}}. \]
We recall that the constant $S$ appearing in \cite{BCN} (see eq. (2.1) therein) corresponds to our $S^2_{1, \frac{N+4}{N-4}}$. 

We now prove that $\Sigma$ is attained. We borrow some ideas from \cite[Lemma 3.2]{BCN}.
We preliminary notice that 
\[ u\mapsto \left( \int_\Omega \abs{\Delta u}^2 + \int_\Omega |u|^2 \right)^{1/2} \]
is a norm in $H_\nu^2$. Indeed,
let $u \in H_\nu^2$, and consider the problem
\[ 
\Delta w= \Delta u \text{ in } \Omega, \qquad \partial_\nu w=0 \text{ on } \partial \Omega,\qquad 
\int_\Omega w= 0 
\]
This problem is well defined as by the divergence theorem and since $\partial_\nu u =0$ on $\partial \Omega$, 
\[ \int_\Omega \Delta u= \int_{\partial \Omega} \partial_\nu u=0. \]
Therefore
\[ w:= u- \frac{1}{|\Omega|} \int_\Omega u \]
is the unique solution of this problem and, by regularity theory (Lemma \ref{lemma:regularity}), $\norm{w}_{H^2} \le C \norm{\Delta u}_2,$
whence $\norm{u}_{H^2} \le C \left( \norm{u}_2 + \norm{\Delta u}_2 \right).$

Now, take $u_k$ a minimizing sequence for $\Sigma$. Then, 
\[ \int_\Omega \abs{\Delta u_k}^2 \le C \quad \text{ and also } \quad  \int_\Omega \abs{u_k}^2 \le |\Omega|^\frac{4}{N} \left(\int_\Omega \abs{u_k}^{\frac{2N}{N-4}}\right)^\frac{N-4}{N} =|\Omega|^\frac{4}{N}, \]
hence in particular $u_k$ is bounded in $H_\nu^2$. Therefore, up to extracting a subsequence, there exists $u \in H_\nu^2$ such that 
\[
 \text{$u_k \to u$ a.e. in $\Omega$, weakly in $L^2$, strongly in $H^1$;}\quad \text{$\Delta u_k \rightharpoonup \Delta u$ weakly in $L^2$.  }
\]

\textit{Step one: $u \not \equiv 0$.} Assume by contradiction $u \equiv 0$, then $u_k \to 0$ strongly in $H^1(\Omega)$. Thus, by Cherrier's inequality \eqref{cherrier} with $\eta=2$, we have that for every $\eps>0$,
\[  1 =\|u\|_{\frac{2N}{N-4}}  \le \left( \frac{2^{\frac 2N}}{S_{1, \frac{N+4}{N-4}}} + \eps \right) \norm{\Delta u_k}_2 + o(1)=\left( \frac{2^{\frac 2N}}{S_{1, \frac{N+4}{N-4}}} + \eps \right) \Sigma + o(1).\]
 This contradicts \eqref{dis sigma} for sufficiently small $\eps$ and large $k$, and allows us to conclude that $u \not \equiv 0$.

\textit{Step two: strong convergence in $L^{\frac{2N}{N-4}}$.} 
One has
\begin{equation}\label{brezislieb} 1- \int_\Omega |u_k -u|^{\frac{2N}{N-4}} = \int_\Omega |u|^{\frac{2N}{N-4}} + o(1),\end{equation}
by the Brezis-Lieb lemma, see for instance \cite[Lemma 1.32]{Willem}. 
Since $\Delta u_k \rightharpoonup \Delta u$ weakly in $L^2$, 
\begin{equation}\label{decomposition} \int_\Omega \abs{\Delta u_k}^2= \int_\Omega \abs{\Delta(u_k-u)}^2 + \int_\Omega \abs{\Delta u}^2 + 2 \int_\Omega \Delta(u_k - u) \Delta u= \int_\Omega \abs{\Delta(u_k-u)}^2 + \int_\Omega \abs{\Delta u}^2 + o(1). \end{equation}
Also, since $u\not \equiv 0$ by the previous step and by the definition of $\Sigma$,
\begin{equation}\label{stima sigma} \int_\Omega \abs{\Delta u}^2 \ge \Sigma \norm{u}_{\frac{2N}{N-4}}^2. \end{equation}
Therefore, by exploiting \eqref{decomposition}, \eqref{cherrier}, the fact that $\frac{N-4}{N}$<1, \eqref{stima sigma}, and \eqref{brezislieb}, 
\begin{align*}
\Sigma &= \int_\Omega \abs{\Delta(u_k-u)}^2 + \int_\Omega \abs{\Delta u}^2 + o(1) \ge \left( \frac{2^{\frac 4N}}{S_{1, \frac{N+4}{N-4}}^2} + \eps \right)^{-1} \norm{u_k-u}_{\frac{2N}{N-4}}^2 +  \int_\Omega \abs{\Delta u}^2 + o(1) \\
& \ge  \left( \frac{2^{\frac 4N}}{S_{1, \frac{N+4}{N-4}}^2} + \eps \right)^{-1} \norm{u_k-u}_{\frac{2N}{N-4}}^2 +  \Sigma \norm{u}_{\frac{2N}{N-4}}^2 + o(1) \\
& = \left[  \left( \frac{2^{\frac 4N}}{S_{1, \frac{N+4}{N-4}}^2} + \eps \right)^{-1} - \Sigma \right] \norm{u_k-u}_{\frac{2N}{N-4}}^2 + \Sigma  \left( \norm{u_k-u}_{\frac{2N}{N-4}}^2 + \norm{u}_{\frac{2N}{N-4}}^2 \right) + o(1) \\
& \ge \left[  \left( \frac{2^{\frac 4N}}{S_{1, \frac{N+4}{N-4}}^2} + \eps \right)^{-1} - \Sigma \right] \norm{u_k-u}_{\frac{2N}{N-4}}^2 + \Sigma  \left( \norm{u_k-u}_{\frac{2N}{N-4}}^{\frac{2N}{N-4}} + \norm{u}_{\frac{2N}{N-4}}^{\frac{2N}{N-4}} \right)^{\frac{N-4}{N}} + o(1) \\
& =  \left[  \left( \frac{2^{\frac 4N}}{S_{1, \frac{N+4}{N-4}}^2} + \eps \right)^{-1} - \Sigma \right] \norm{u_k-u}_{\frac{2N}{N-4}}^2 + \Sigma + o(1).
\end{align*}
By taking $\eps>0$ sufficiently small so that $ \left(\frac{2^{\frac 4N}}{S_{1, \frac{N+4}{N-4}}^2} + \eps\right)^{-1}-\Sigma>0$, we deduce by \eqref{dis sigma} that $u_k \to u$ strongly in $L^{\frac{2N}{N-4}}$, thus $\norm{u}_{\frac{2N}{N-4}}=1$. 

\textit{Step three: conclusion of the proof.}
By weak lower semicontinuity of $\norm{\cdot}_2$, one has
\[ \Sigma \le \int_\Omega \abs{\Delta u}^2 \le \liminf_{k\to \infty} \int_\Omega \abs{\Delta u_k}^2 = \Sigma, \]
hence $\Sigma$ is attained at $u$. 
 \end{proof}

We are now ready to prove Theorem \ref{main thm1.2}.  
\begin{proof}[Proof of Theorem \ref{main thm1.2}]
We show that  $(i)$-$(ii)$ in Lemma \ref{pq close} hold with $p^*=1$ and $q^*=\frac{N+4}{N-4}$. 
We start by proving condition $(i)$. 
Since $\Sigma$ is attained by Lemma \ref{sigma attained}, there exists $\bar u \in H^2_\nu(\Omega)$ such that 
\[ \int_\Omega |\Delta \bar u|^2= \Sigma, \quad \int_\Omega |\bar u|^{\frac{2N}{N-4}} =1, \]
and $\bar u$ satisfies 
\[
\Delta^2 \bar u= |\bar u|^{\frac{8}{N-4}} \bar u \text{ in } \Omega,\qquad  
\partial_\nu \bar u=\partial_\nu (\Delta \bar u)=0 \text{ on } \partial \Omega.  \]
Take 
\[ f=- \Delta \bar u, \quad g=|\bar u|^{\frac{8}{N-4}}\bar u. \]
One has $(f, g) \in X$, and 
\[ \int_\Omega g K f= \int_\Omega \bar u^{\frac{2N}{N-4}} =1,\qquad \norm{f}_2= \Sigma^{\frac 12},\qquad \text{ and } \norm{g}_{\frac{2N}{N+4}}= 
\left( \int_\Omega |\bar u|^{\frac{2N}{N-4}}  \right)^{\frac{N+4}{2N}}=1. \]
Thus, by the definition of $D_{1,\frac{N+4}{N-4}}$ and \eqref{dis sigma},
\[ D_{1, \frac{N+4}{N-4}} \ge \frac{\int_\Omega g K f}{\norm{f}_2\norm{g}_{\frac{2N}{N+4}} } = \Sigma ^{-\frac 12}> \frac{2^{\frac 2N}}{S_{1, \frac{N+4}{N-4}}}. \]

Let us now prove condition $(ii)$ in Lemma \ref{pq close}. Recall the bubbles defined in \eqref{eq:familybubbles}. Here, $\eps$ and $x_0$ are not important in the analysis, hence we will just fix $\eps=1$ and $x_0=0$. However, the dependence on $p, q$ (which solve \eqref{CH}) will be crucial. Thus, in what follows, we assume without loss of generality that $q\geq \frac{N+4}{N-4}$ and we will denote by $(U_p, V_p)$ the solution to
\[ 
    -\Delta U_p=V_p^q \text{ in } \R^N,\qquad
    -\Delta V_p=U_p^p  \text{ in } \R^N,\qquad
    U_p(0)=1. \]
By Proposition 2.2 in \cite{Pistoia}, we know that
\[ (U_p, V_p) \to (U_1, V_1)  \text{ in } \mathcal{D}^{2,2}(\R^N) \times \mathcal{D}^{2,\frac{2N}{N+4}}(\R^N)\qquad \text{ as $p \to 1$.} \]

Since $U_p \to U_1$ in $\mathcal{D}^{2,2}(\R^N)$, then up to a subsequence we have $U_p \to U_1$ in  $L^{\frac{2N}{N-4}}(\R^N)$. In particular, 
by the inverse of the dominated convergence theorem, we get a function $h \in L^{\frac{2N}{N-4}}(\R^N)$ such that $\abs{U_p} \le h$.
Therefore, using the fact that the sequence $U_p$ is uniformly bounded (by 1), and since $q+1\geq \frac{2N}{N-4}$, there exists a constant $C>0$ such that 
$ \abs{U_p}^{q+1} \le C \abs{U_p}^{\frac{2N}{N-4}} \le C h^{\frac{2N}{N-4}}$. Whence, as $p \to 1$, by dominated convergence
\[ S_{p, q}^\frac{N}{2}=\int_{\R^N} U_p^{q+1} \to \int_{\R^N} U_1^{\frac{2N}{N-4}}=S_{1, \frac{N+4}{N-4}}^\frac{N}{2}. \]

Now, Lemma \ref{pq close} applies with $p^*=1$ and $q^*=\frac{N+4}{N-4}$, thus
\[ D_{p, q} > \frac{2^{\frac 2N}}{S_{p, q}} \]
for $(p, q)$ close enough to $(1, \frac{N+4}{N-4})$. 
The conclusion now easily follows recalling  Lemma \ref{condition D}. Proposition \ref{reg2} yields regularity of solutions. 
\end{proof}

\section{Symmetry-breaking and solutions with radial symmetry. Proof of Theorem \ref{thm:annulus}}\label{sec:symm}

In this section we restrict our attention to the case
\begin{equation}\label{omega simm} \begin{aligned}
    \Omega&=B_R(0)\setminus B_r(0), \qquad \text{ for some } 0<r<R \quad \text{ or }\\
    \Omega&=B_R(0) \qquad \text{ for some } R>0.
\end{aligned} \end{equation}
The proof of Theorem \ref{thm:annulus} follows more or less directly from the ideas in \cite[Theorems 1.1--1.2]{ST2}, where the subcritical case is treated, combined with the regularity result in Proposition \ref{reg2}. Here we just highlight the differences, presenting a sketch of the proof. We start by defining some radially symmetric counterparts of notions introduced before.

\begin{definition}\label{def:radial}
We define
\[ X_{rad}=\{ (f, g) \in X: f, g \text{ are radially symmetric } \} \]
and
\[ D_{rad}=\sup \left \{\int_\Omega f Kg: \quad  (f, g) \in X_{rad}, \quad \gamma_1 \norm{f}_\alpha^\alpha + \gamma_2 \norm{g}_\beta^\beta =1 \right \}.
\]
\end{definition}

The key point of the proof of Theorem \ref{thm:annulus} (and of \cite[Theorem 1.2]{ST2}) is based on an $L^t$--norm-preservation transformation introduced in \cite{ST2}, which we now recall. Let 
\begin{align*}
&\mathcal{I}:L_{rad}^\infty(\Omega)\to C_{rad}(\overline{\Omega}),\qquad {\cal I}h(x):=\int_{r \leq|y|{\leq}|x|\}} h(y)\ dy
=N\omega_N\int_{r}^{|x|}h(\rho)\rho^{N-1}\ d\rho\\
&\mathfrak{F}: C_{rad}(\overline{\Omega})\to L_{rad}^\infty(\Omega),\qquad \mathfrak{F} h:=(\chi_{\{{\cal I}h>0\}}-\chi_{\{{\cal I}h\leq 0\}})\, h.
\end{align*}
\begin{definition}
For $h\in C_{rad}(\overline{\Omega})$, the $\divideontimes$-transformation is given by 
\begin{align*}
h^{\divideontimes}\in L^\infty_{{rad}}(\Omega),\qquad h^{\divideontimes}(x) := ({\mathfrak F} h)^\#(\omega_N |x|^N-\omega_N {r}^N), 
\end{align*}
where $\omega_N= |B_1|$ is the volume of the unitary ball in $\R^N$ and $\#$ is the decreasing rearrangement given by 
\begin{align*}
h^\#:[0,|\Omega|]\to\R,\qquad h^{\#}(0):=\text{ess sup}_{\Omega} h\quad h^\#(s):=\inf\{t\in\R\::\: |\{h>t\}|<s\},\ s>0.
\end{align*}
\end{definition}
This transformation, in practice, is applied to radial functions $h$ with zero average: $\mathcal{I}h(R)=\int_\Omega h=0$. Loosely speaking, this transformation does the following: in many situations the domain of $h$ may be split in $r=:r_0<r_1<\ldots<r_N=R$, where $\int_{r_i<|x|<r_{i+1}}h=0$; however it may not be true that $\mathcal{I}h(x)$ is nonnegative for every $x$. We flip the graph of $h$ in the annuli $[r_i,r_{i+1}]$ where we don't have this property, obtaining at the end $\mathcal{I}(\mathfrak{F} h)(x)\geq 0$ for every $x\in \Omega$. We then apply a decreasing rearrangement, finally placing the result back in the original annulus using the transformation $x\mapsto \omega_N|x|^N-\omega_N r^N$. For more details, insights, examples and comments regarding the definition of the  \emph{flip-\&-rearrange transformation} $\divideontimes$ we refer to \cite[Section 3.2]{ST2}. The following is a combination of Theorem 1.3 and Proposition 3.4 from \cite{ST2}.

\begin{theorem}\label{thm:trans} Let $p,q>0$ and take $\Omega$ as in \eqref{omega simm}.
Take  $f,g:\overline{\Omega}\to \R$ be continuous and radially symmetric functions with $\int_\Omega f = \int_\Omega g= 0$. Then $(f^\divideontimes,g^\divideontimes)\in X_{rad}$,
\begin{align}\label{in}
\|f^\divideontimes\|_\frac{p+1}{p}=\|f\|_\frac{p+1}{p}\quad \|g^\divideontimes\|_\frac{q+1}{q}=\|g\|_\frac{q+1}{q} \text{ and }\quad  \int_\Omega fK g\leq \int_\Omega f^\divideontimes K g^\divideontimes.
\end{align}
Furthermore, if $f,g$ are nontrivial and the last statement in \eqref{in} holds with equality, then $f,g$ are monotone in the radial variable. Moreover, $(Kf,Kg)$ is radially symmetric and $(Kf)_r(Kg)_r>0$.
\end{theorem}

We also recall some standard notation and the definition of foliated Schwarz symmetry. 
Let $N \ge 4$, and fix $e \in \mathbb{S}^{N-1}$. We consider the half-space $H(e)=\{ x \in \R^N : x\cdot e >0 \}$, and the half domain $\Omega(e)=\{ x \in \Omega : x\cdot e >0 \}$. Given a function $w: \overline \Omega \to \R$ we define $w_e: \overline \Omega \to \R$ as
\[ w_e(x)=w(x-2(x\cdot e)e). \]
The polarization $w^H$ of $w$ with respect to a hyperplane $H=H(e)$ is given by
\[ w^H= \begin{cases}
\max \{ w, w_e \} & \text{ in } \overline{\Omega(e)} \\
\min \{ w, w_e \} & \text{ in } \overline \Omega \setminus \overline{\Omega(e)}.
\end{cases} \]
We say that $u \in C(\overline \Omega)$ is foliated Schwarz symmetric with respect to some unit vector $e^* \in \mathbb{S}^{N-1}$ if $u$ is axially symmetric with respect to the axis $\R e^*$ and non increasing in the polar angle $\theta:= \arccos(\frac{x}{|x|} \cdot e^*) \in [0, \pi]$. 
The following characterization of foliated Schwarz symmetry holds, see \cite{SW}.
\begin{lemma}\label{lem sym}
There is $e^* \in \mathbb{S}^{N-1}$ such that $ u, v \in C(\overline \Omega)$ are foliated Schwarz symmetric with respect to  $e^*$ if and only if for each $e \in \mathbb{S}^{N-1}$ either
\[ u \ge u_e, v \ge v_e \quad \text{ in } \Omega(e) \qquad \text{ or } \qquad u \le u_e, v \le v_e \quad \text{ in } \Omega(e). \]
\end{lemma}

We are now ready to prove Theorem \ref{thm:annulus}. 
\begin{proof}[Proof of Theorem \ref{thm:annulus}]
Proof of Part (i). By our assumption, $\Omega$ is an annulus. We first observe that $D_{rad} >0$, as follows by taking $(\varphi_1, \varphi_1)$, where $\varphi_1$ is the first non constant radial eigenfunction of the Laplacian with Neumann boundary conditions. Since $\Omega$ is an annulus, for any $t>1$ the embedding
 \begin{equation}\label{eq:embedding_radial}
 W^{2,t}_{rad}(\Omega)\subset W^{1,t}_{rad}(\Omega)\hookrightarrow C(\overline \Omega)
 \end{equation}
 is compact, therefore, also by Lemma \ref{lemma:regularity},  the operator $K$ is compact from $L^{\frac{q+1}{q}}_{rad}(\Omega)$ to $C(\overline \Omega)$. In particular, $K$ is compact from  $L^{\frac{q+1}{q}}_{rad}(\Omega)$ to  $L^{p+1}_{rad}(\Omega)$. With this, we can prove that $D_{rad}$ is achieved at a certain pair $(f,g)$: indeed, taking a maximizing sequence $(f_k,g_k)\in X_{rad}$ such that $\int_\Omega f_kKg_k\to D_{rad}$ and $\gamma_1 \norm{f_k}_\alpha^\alpha + \gamma_2 \norm{g_k}_\beta^\beta =1 $, up to a subsequence we have
 \[
 f_k\rightharpoonup f \text{ weakly in } L^{\frac{p+1}{p}}(\Omega),\qquad g_k\rightharpoonup g \text{ weakly in } L^{\frac{q+1}{q}}(\Omega),\qquad Kg_k \to Kg \text{ strongly in } L^{p+1}(\Omega).
 \]
 Then $\gamma_1 \norm{f}_\alpha^\alpha + \gamma_2 \norm{g}_\beta^\beta\leq 1$ and $\int_\Omega fKg=D_{rad}$, so that $(f,g)\neq (0,0)$ and (\emph{cf.} Lemma \ref{lemma:equivalent_def_D})
 \[
 \frac{D_{rad}}{(\gamma_1 \norm{f_k}_\alpha^\alpha + \gamma_2 \norm{g_k}_\beta^\beta)^\frac{1}{\gamma}} =\int_\Omega \frac{f}{(\gamma_1\|f\|_\alpha^\alpha+\gamma_2\|g\|_\beta^\beta)^\frac{1}{\alpha}} K \left(\frac{g}{(\gamma_1\|f\|_\alpha^\alpha+\gamma_2\|g\|_\beta^\beta)^\frac{1}{\beta}}\right)\leq D_{rad},
 \]
 so that $\gamma_1 \norm{f}_\alpha^\alpha + \gamma_2 \norm{g}_\beta^\beta=1$, and $(f,g)$ achieves $D_{rad}$.
 
 Therefore, reasoning precisely as in the proof of Proposition \ref{prop:equiv}, 
 \begin{equation}\label{eq:relation_Drad_with_Lrad}
 (u, v)=((D_{rad})^{-q \frac{p+1}{pq-1}} K_p g, (D_{rad})^{-p \frac{q+1}{pq-1}} K_q f)=(D_{rad}^{-\frac{q+1}{pq-1}} |f|^{\frac 1 p- 1} f, D_{rad}^{-\frac{p+1}{pq-1}} |g|^{\frac 1 q- 1} g)
 \end{equation}
 is a least-energy radial solution of \eqref{system}.  
 
 From the second relation in \eqref{eq:relation_Drad_with_Lrad}, we have that $f, g$ are also continuous up to the boundary. From Theorem \ref{thm:trans} (using the maximality of $(f,g)$) and by the first relation in \eqref{eq:relation_Drad_with_Lrad}, we now conclude that $u_rv_r>0$.
 
 %In particular $(u,v)\in W_{{rad}}^{2,\frac{p+1}{p}}(\Omega)\times W_{{rad}}^{2,\frac{q+1}{q}}(\Omega)\subset  L^\infty(\Omega)\times L^\infty(\Omega)$ by \eqref{eq:embedding_radial}. Again by elliptic regularity (Lemma \ref{lemma:regularity}) we have $u,v\in W^{2,t}(\Omega)$ for every $t>1$. By Sobolev embeddings, then $u,v\in C^{0,\gamma}(\overline \Omega)$ for every $\gamma\in (0,1)$. Finally, since the maps $t\mapsto |t|^{p-1}t,|t|^{q-1}t$ are H\"older continuous, by Schauder estimates (\cite[Theorem 6.31]{GT} and remark after the theorem) we have $u,v\in C^{2,\eps}(\overline \Omega)$ for every $\eps\in (0,1)$.
 
Proof of Part (ii): Under the assumptions of the statement, we have that $D_{p,q}$ is achieved. By simple scalings (recall the proof of Proposition \ref{prop:equiv}), we have that
 \[
 D_{p,q}^{-\frac{N}{2}}\frac{2}{N}=\tilde D_{p,q}:=\inf_\mathcal{N} \phi(f,g),
 \]
  where  $\mathcal{N}=\{(f,g)\in X\setminus\{(0,0)\}: \ \phi'(f,g)(\gamma_1f,\gamma_2g)=0\}$ is the standard Nehari manifold, and $\tilde f=(D_{p,q})^s f$, and $\tilde g=(D_{p,q})^t g$, with $s=-p\frac{q+1}{pq-1}$ and $t=-q\frac{p+1}{pq-1}$ achieves $\tilde D_{p,q}$. 
Then we are precisely in the assumptions of \cite[Theorem 4.1]{ST2} (whose proof does not depend on the value of $(p,q)$ with respect to the critical hyperbola, if we take into account Theorem \ref{thm:trans} and the regularity result in Proposition \ref{reg2}). 
Then $\tilde f, \tilde g$ are not radially symmetric, and neither are $f,g$.

We finally show that $u, v$ are foliated Schwarz symmetric. 
Fix a hyperplane $H=H(e)$ for some $|e|=1$. Notice that  (see Proposition \ref{prop:equiv})
\[ (u, v)=(D_{p, q}^{-q\frac{p+1}{pq-1}} \, K_p g, \; D_{p,q}^{-p\frac{q+1}{pq-1}} \,K_q f) \]
is a solution of \eqref{system}, and in particular $-\Delta u= D_{p,q}^{-q\frac{p+1}{pq-1}} g$ and $-\Delta v=D_{p,q}^{-p\frac{q+1}{pq-1}} f$. 
Define 
\[ (\tilde u, \tilde v)=(D_{p,q}^{-q\frac{p+1}{pq-1}} \, K_p(g^H), \; D_{p,q}^{-p\frac{p+1}{pq-1}} \,K_q(f^H)). \]
By Proposition \ref{reg2}, we know that $(u, v) \in C^{2, \zeta}(\overline \Omega) \times C^{2, \eta}(\overline \Omega)$ with suitable $\zeta, \eta$, and moreover $f^H, g^H \in L^\infty(\Omega)$ and $(\tilde u, \tilde v) \in (W^{2, p}(\Omega) \cap C^{1, \zeta}(\overline \Omega)) \times (W^{2, q}(\Omega) \cap C^{1, \eta}(\overline \Omega))$. 
Let $V=v_e + v- \tilde v -\tilde v_e$, thus 
\[ -\Delta V= D_{p,q}^{-p \frac{p+1}{pq-1}}  \, (f- f^H - (f^H)_e + f_e)=0 \text{ in } \Omega, \quad \partial_\nu V=0 \text{ on } \partial \Omega. \]
Testing this equation with $V$ and integrating by parts we obtain that $V=k$ for some $k \in \R$. Then
\begin{equation}\label{V=k}
v_e + v= \tilde v_e + \tilde v +k \quad \text{ in } \Omega. 
\end{equation}
Let 
\[ \Gamma_1= \{ x \in \partial \Omega(e): x \cdot e =0 \}, \qquad \Gamma_2= \{ x \in \partial \Omega(e): x \cdot e > 0 \}, \]
$w_1= \tilde v - v +k/2$, and $w_2= \tilde v- v_e + k/2$. Since $v=v_e$ and $\tilde v=\tilde v_e$ on $\Gamma_1$, we have that $w_1=w_2=0$ on $\Gamma_1$, and $\partial_\nu w_1=\partial_\nu w_2=0$ on $\Gamma_2$. 
Also,
\[ -\Delta w_1= D_{p,q}^{-p \frac{p+1}{pq-1}} (f^H - f )\ge 0 \quad \text{ in } \Omega(e) \qquad -\Delta w_2= D_{p,q}^{-p \frac{p+1}{pq-1}}(f^H-f_e) \ge 0 \quad \text{ in } \Omega(e), \]
which implies that $w_1 \ge 0$ and $w_2 \ge 0$ in $\Omega(e)$. 
Therefore, since by \eqref{V=k} one has $\tilde v_e= v_e + v - \tilde v - k$, and by definition of $g^H$ one has $g_e^H=g_e+g-g^H$, and also recalling $\int_\Omega g=0$, 
\begin{align}\label{stima sim}
\nonumber D^{-p\frac{q+1}{pq-1}} \int_\Omega (g K f - g^H K f^H) &=\int_\Omega (g v - g^H \tilde v) = \int_{\Omega(e)} gv+ g_e v_e - g^H \tilde v - (g^H)_e \tilde v_e\\
\nonumber  &=\int_{\Omega(e)} gv+ g_e v_e -g^H \tilde v - (g_e+g-g^H)(v_e + v - \tilde v - k) \\
&=\int_{\Omega(e)}  (g_e-g^H)w_1 + (g-g^H) w_2 + \frac k2 (g_e +g) \le 0.
\end{align}
To show that $u, v$ are foliated Schwarz symmetric, we now use Lemma \ref{lem sym}, and argue by contradiction. Then, we assume that there is $e \in \mathbb{S}^{N-1}$ and the corresponding half space $H=H(e)$ such that $u \ne u^H$ in $\Omega(e)$ and either $v_e \ne v^H$ in $\Omega(e)$ or $u_e \ne u^H$ in $\Omega(e)$. Since $t \mapsto  c\, |t|^s t$ is a strictly monotone increasing function in $\R$ for $s > -1$, and since $f= D_{p,q}^{-p\frac{p+1}{pq-1}} |u|^{p-1} u$, 
then 
\[ f \ne f^H \qquad \text{ and either } f_e \ne f^H \quad \text{ or } g_e \ne g^H \]
in $\Omega(e)$. 
Therefore, $w_1>0$ and either $w_2 >0$ or $0 \ne g_e - g^H \le 0$ in $\Omega(e)$.   
Also, notice that if $f \ne f^H$ in $\Omega(e)$, then $g \ne g_e$ in $\Omega(e)$, namely either $g - g^H\ne0$ or $g_e - g^H \ne 0$. Summing up, we have two possibilities: either $w_1>0$ and $0 \ne g_e - g^H \le 0$; or $w_2 >0$ and $0 \ne g - g^H \le 0$. In both cases,
recalling now \eqref{stima sim}, we conclude
\[ \int_\Omega g Kf < \int_\Omega g^H K f^H, \]
which contradicts the maximality of $(f, g)$. 
\end{proof}

\begin{remark}\label{rmk:pohoz}
Notice that, when $p=q$, the system reduces to a  single equation. Then, the fact that least energy solutions are non radial in a ball $\Omega=B_R(0)$ is a simple consequence of the Pohozaev identity: if a radial solution existed, since $u$ is sign-changing  we would have $u=0$ on $\partial B_\rho(0)$ for some $0<\rho<R$, and we would have a positive solution of 
\[
-\Delta u=u^\frac{N+2}{N-2} \in B_{\rho}(0),\quad u=0 \text{ on } \partial B_\rho(0),
\]
a contradiction. Such argument does not extend to the case $p\neq q$; even though a Pohozaev-type  identity holds (see \cite{ST2}), if $(u,v)$ is a least energy solution it is not clear if both $u$ and $v$ vanish on the same sphere, so we do not obtain, in principle, a solution to a Dirichlet problem in a smaller ball.
\end{remark}

\appendix
\section{Cherrier's-type inequalities}\label{app:A}
For $\eta> 1$ and  $N>2\eta$, we define $S$ to be the best Sobolev constant of the embedding $\mathcal{D}^{2,\eta}(\R^N)\hookrightarrow L^{\frac{N\eta}{N-2\eta}}(\R^N)$, that is 
\[
S=\inf \left\{  \norm{\Delta u}_{L^\eta(\R^N)} :\ u\in \mathcal{D}^{2,\eta}(\R^N),\ \int_{\R^N} |u|^\frac{N\eta}{N-2\eta}=1 \right\},
\]
$S \|u\|_{L^{\frac{\eta N}{N-2\eta}}(\R^N)}\leq  \|\Delta u\|_{L^\eta(\R^N)}$. With our notation above, we have 
\[ S= S_{\frac{\eta N - N+2\eta}{N-2\eta},\frac{1}{\eta-1}}.\]

The main purpose of this appendix is to prove the following Cherrier inequality. We would like to point out that this one is not included in Cherrier's original paper \cite{Cherrier}. We are inspired by the proof of \cite[Lemma 1.2]{BCN}, which deals with the case $\eta=2$ and $N\geq 5$.

\begin{theorem}\label{thm:Cherrier} Let $\eta> 1$ and $N > 2\eta$. For any $\varepsilon>0$ there exists a constant $C(\varepsilon)>0$ such that, for any $u \in W^{2, \eta}_\nu(\Omega):=\{u \in W^{2, \eta}(\Omega): u_\nu=0 \text{ on } \partial \Omega \}$, one has 
\begin{equation}\label{cherrier} \norm{u}_{\eta^*} \le  \left (\frac{2^{\frac2N}}{S} +\varepsilon\right ) \norm{\Delta u}_{\eta} + C(\varepsilon) \norm{u}_{W^{1,\eta}}, \end{equation}
where $\eta^*=\frac{N\eta}{N-2\eta}$.
\end{theorem} 
\begin{proof}
Following the proof of (1.8) of Lemma 2.1 in \cite{BCN} (see Step two therein), we choose finitely many points $x_i \in \overline \Omega$ and a positive number $R >0$, with corresponding sets $\Omega_i= \Omega \cap B_{R}(x_i)$, such that $ \overline \Omega \subset \bigcup_{i=1}^{n} \Omega_i$. Up to increasing the number of open sets, we can assume that $x_i \in \partial \Omega$ whenever $\overline \Omega_i \cap \partial \Omega \ne \emptyset$. 
We choose $(\tilde \zeta_i)_{i=1}^n$ a smooth partition of the unity subordinated to the covering $(\Omega_i)_{i=1}^n$:
\[
\tilde \zeta_i\in C^\infty(\R^N),\quad  0\leq \tilde \zeta_i \leq 1,\quad  \textrm{supp}\, \tilde \zeta_i \subset \Omega_i, \quad \sum_{i=1}^n \tilde \zeta_i(x)=1 \ \forall x\in \Omega.
\]
 We split the set of indices into two disjoint sets $\mathcal{I}$ and $\mathcal{J}$, 
where $\mathcal{I}$ contains the indices with $x_i \in \Omega$ (and $\Omega_i\subset \Omega$), while $\mathcal{J}$ contains the indices with $x_i \in \partial \Omega$ ($\cup_{i\in \mathcal{J}} \Omega_i$ covers the boundary). 
We define 
\[ \zeta_i = \frac{ \tilde \zeta_i^{2\eta+1}}{\sum_{i=1}^n \tilde \zeta_i^{2\eta+1}}. \]
Notice that this is a partition of unity subordinated to the same covering, and that $\zeta_i^{1/\eta} \in C^2(\overline \Omega)$. 
Then, for any $u \in W^{2, \eta}(\R^N)$, one has $\zeta_i^{1/\eta} u \in W^{2, \eta}(\R^N)$ and $\text{supp}(\zeta_i^{1/\eta} u)\subset \Omega_i$.
We consider separately the indices in $\mathcal{I}$ and in $\mathcal{J}$. 

\textit{Case One.}
For any $i \in \mathcal{I}$  and for any $\varepsilon>0$ there exists a constant $B_1(\varepsilon)$ such that 
\[  \sum_{i \in \mathcal{I}}  \left( \int_{\Omega_i }\abs{\zeta_i^{1/\eta} u}^{\eta^*} \right)^{\frac{\eta}{\eta*}} \le  \left( \frac{2^{\frac{2}N}}{S} +\varepsilon \right)^\eta
\left(\sum_{i \in \mathcal{I}} \int_{\Omega} \zeta_i \abs{\Delta u}^\eta \right) + B_1(\varepsilon) \norm{u}_{W^{1, \eta}}^\eta. \]

Indeed,
\begin{align*}
 \sum_{i \in \mathcal{I}}  \left( \int_{\Omega_i }\abs{\zeta_i^{1/\eta} u}^{\eta^*} \right)^{\frac{\eta}{\eta*}} 
 &\le  \sum_{i \in \mathcal{I}}  \left( \int_{\R^N}\abs{\zeta_i^{1/\eta} u}^{\eta^*} \right)^{\frac{\eta}{\eta*}} \le \frac{1}{S^\eta} \sum_{i \in \mathcal{I}}\int_{\R^N} \abs{\Delta(\zeta_i^{1/\eta} u)}^\eta\\
 &= \frac{1}{S^\eta} \sum_{i \in \mathcal{I}} \left[ \int_{\R^N} \left(\zeta_i^{1/\eta} |\Delta u| + 2 |\nabla (\zeta_i^{1/\eta})||\nabla u| + |u||\Delta (\zeta_i^{1/\eta})|\right)^\eta  \right] \\
 &\le \frac{2^{\frac{2\eta}N}}{S^\eta} \sum_{i \in \mathcal{I}} \left[ \int_{\R^N} \left(\zeta_i^{1/\eta} |\Delta u| + 2 |\nabla (\zeta_i^{1/\eta})||\nabla u| + |u||\Delta (\zeta_i^{1/\eta})|\right)^\eta  \right],
\end{align*} 
where in the last inequality we just used the fact that $2^{\frac{2\eta}N} \ge 1$. Let us call
\[ X= \zeta_i^{1/\eta} |\Delta u|, \quad Y= 2 |\nabla (\zeta_i^{1/\eta})||\nabla u|,  \quad Z=|u||\Delta (\zeta_i^{1/\eta})|. \]
Recall that, given $\varepsilon_0>0$, there exists $C(\varepsilon_0)$ such that $(1+t)^\eta\leq (1+\varepsilon_0)+C(\varepsilon_0) t^\eta$ for every $t>0$.
This, in turn, implies that $(s+t)^\eta \leq s^\eta(1+\varepsilon_0)+C(\varepsilon_0)t^\eta$.
Hence:
\begin{equation}\label{eq:inequality_XYZ}
(X+Y+Z)^\eta\leq (1+\varepsilon_0)X^\eta+C(\varepsilon_0)(Y+Z)^\eta\leq (1+\varepsilon_0)X^\eta+C(\varepsilon_0,\eta)(Y^\eta+Z^\eta).
\end{equation}
Notice that
\[ \int_{\R^N}  X^\eta= \int_{\Omega} \zeta_i \abs{\Delta u}^\eta, \quad \text{ whereas } \quad
 \int_{\R^N} (Y^\eta+Z^\eta) \le \int_{\Omega} (\abs{u}^\eta + |\nabla u|^\eta). \]
Thus
\[  \sum_{i \in \mathcal{I}}  \left( \int_{\Omega_i }\abs{\zeta_i^{1/\eta} u}^{\eta^*} \right)^{\frac{\eta}{\eta*}}  \le   \frac{2^{\frac{2\eta}N}}{S^\eta} (1+\varepsilon_0) \left( \sum_{i \in \mathcal{I}} \int_{\Omega} \zeta_i \abs{\Delta u}^\eta \right)  + C_2(\varepsilon_0) \norm{u}_{W^{1, \eta}}^\eta, \]
from which the conclusion follows, as 
\[  \frac{2^{\frac{2\eta}N}}{S^\eta} (1+\varepsilon_0)  \le \left( \frac{2^{\frac2N}}{S} +\varepsilon \right)^\eta \]
if $\varepsilon_0$ is small enough.

\textit{Case Two.}
We claim that, for any $i \in \mathcal{J}$  and for any $\varepsilon>0$, there exists a constant $B_2(\varepsilon)$ such that 
\[ \sum_{i \in \mathcal{J}}  \left( \int_{\Omega_i \cap \Omega} \abs{\zeta_i^{1/\eta} u}^{\eta^*} \right)^{\frac{\eta}{\eta*}} \le \left( \frac{2^{\frac2N}}{S} +\varepsilon \right)^\eta\left( \sum_{i \in \mathcal{J}} \int_{\Omega} \zeta_i \abs{\Delta u}^\eta\right) + B_2(\varepsilon) \norm{u}_{W^{1, \eta}}^\eta. \]

We first notice that if $u \in W^{2, \eta}_\nu(\R^N_+)$ then 
\begin{equation}\label{ext} \norm{u}_{L^{\eta^*}(\R^N_+)}  \le \frac{2^{\frac2N}}{S}\norm{\Delta u}_{L^\eta(\R^N_+)} \end{equation}
Indeed, let us extend symmetrically $u \in W^{2, \eta}_\nu(\R^N_+)$ to the whole space $\R^N$, reflecting it with respect to the $x_N$ axis. Let us call this extension $\hat u \in W^{2, \eta}(\R^N)$:
\[
\hat u(x):=\begin{cases}
u(x',x_N), & x_N>0,\\
u(x',-x_N), & x_N<0.
\end{cases}
\] 
One has
\[ \norm{\hat u}_{L^{\eta^*}(\R^N)} = 2^{1/{\eta^*}} \norm{u}_{L^{\eta^*}(\R^N_+)}\quad \text{ and }\quad 
 \norm{\Delta \hat u}_{L^\eta(\R^N)} = 2^{1/\eta} \norm{\Delta u}_{L^\eta(\R^N_+)}, \]
and the desired inequality follows using the definition of $S$, and recalling that $W^{2, \eta}(\R^N) \subset \mathcal{D}^{2, \eta}(\R^N)$. 
%see p 29 GGS

We now follow closely \cite[Section 4]{BCN}. 
More precisely, write $x=(x',x_N) \in \R^N$. We know that there exist $R>0$ and a smooth function 
$\rho:\{x' \in \R^{N-1}: |x'|<R \} \to \R_+$ such that (up to a rotation)
\[ 
\Omega \cap B_R = \{(x', x_N) \in B_R: x_N > \rho(x')\},\qquad
\partial \Omega \cap B_R=\{(x', x_N) \in B_R: x_N=\rho(x') \}. 
\]
For any open subset $V$ of $\R^N$ we define 
$\Phi: V \subset \R^N \to \R^N$ such that 
\[ \Phi(y', y_N)=(y', \rho(|y'|))-y_N \nu(y', \rho(|y'|)),
\]
 with  $\nu(x', \rho(x'))=(\nabla \rho(x'), -1)$ an outward orthogonal vector to the tangent space. 
Then the following maps are well defined for some open sets $V_i$
\[ \Phi_i^{-1}=(\Phi^{-1})|_{\Omega_i \cap \Omega} : \Omega_i \cap \Omega \to V_i \subset \R^N_+. \]
Notice that, for any $u \in W^{2, \eta}_\nu(\Omega)$ and  $i \in \mathcal{J}$,  one has $(\zeta_i^{1/\eta} u)\circ \Phi_i \in W^{2, \eta}_\nu(\R^N_+)$. 
Also, we may assume for $\varepsilon_0$ small enough
\[ | \text{det } D \Phi_i(y)| \le 1 + \varepsilon_0, \]
for the details we refer to \cite{BCN}. 
Finally, we call $\theta_i=(\zeta_i^{1/\eta} u) \circ \Phi \in W^{2, \eta}_\nu(\R^N_+)$.

 Then, using \eqref{ext},
\begin{align*} \sum_{i \in \mathcal{J}}  \left( \int_{\Omega_i \cap \Omega} \abs{\zeta_i^{1/\eta} u}^{\eta^*} \right)^{\frac{\eta}{\eta*}} &= \sum_{i \in \mathcal{J}}  \left( \int_{V_i} \abs{\theta_i(y)}^{\eta^*} | \text{det}(D\Phi_i (y))| \right)^{\frac{\eta}{\eta*}}\le (1+\varepsilon_0)^{\frac{\eta}{\eta*}} \sum_{i \in \mathcal{J}} \left( \int_{V_i} |\theta_i(y)|^{\eta^*} \right)^{\frac{\eta}{\eta*}}  \\
&\le \left(\frac{2^{\frac2N}}{S}\right)^{\eta}  (1+\varepsilon_0)^{\frac{\eta}{\eta*}} \sum_{i \in \mathcal{J}} \int_{V_i} |\Delta \theta_i(y)|^\eta.
\end{align*}
Using estimates (47)-(48) in \cite{BCN}, and by applying \eqref{eq:inequality_XYZ} as we did above, we get the conclusion, up to choosing a suitably small $\varepsilon_0$.

\textit{Conclusion. }
Let $u \in W^{2, \eta}_\nu(\Omega)$. We have
\begin{align*}
\norm{u}_{\eta^*}^\eta = \norm{u^\eta}_{\frac{\eta^*}{\eta}} 
& = \norm{\sum_i \zeta_i u^\eta}_{\frac{\eta^*}{\eta}} \le \sum_i \norm{\zeta_i u^\eta}_{\frac{\eta^*}{\eta}} = \sum_i \norm{\zeta_i^{1/\eta} u}_{\eta^*}^\eta \\
&= \sum_{i \in \mathcal{I}} \left( \int_{\Omega_i }\abs{\zeta_i^{1/\eta} u}^{\eta^*} \right)^{\frac{\eta}{\eta*}} +\sum_{i \in \mathcal{J}} \left( \int_{\Omega_i \cap \Omega} \abs{\zeta_i^{1/\eta} u}^{\eta^*} \right)^{\frac{\eta}{\eta*}}.
\end{align*}

Then using the previous steps, 
\[ \norm{u}_{\eta^*}^\eta \le \left( \frac{2^{\frac2N}}{S} +\varepsilon \right)^\eta\left( \sum_{i} \int_{\Omega} \zeta_i \abs{\Delta u}^\eta\right) + B(\varepsilon) \norm{u}_{W^{1, \eta}}^\eta, \]
from which the conclusion follows using the fact that $\zeta_i$ are a partition of unity, taking the power $1/\eta$ and recalling that $(s+t)^{1/\eta} \le  s^{1/\eta} + t^{1/\eta}$ as $\eta \ge 1$.
\end{proof}

\section{Asymptotic estimates}\label{app:B}

\begin{lemma}\label{eq:estimate_normal_derivative}
Assume $p,q$ satisfy \eqref{eq:assumptions_pq} or \eqref{eq:assumptions_pq2}.
One has
\begin{equation}\label{stima norma}
\norm{\partial_\nu \ue}_{L^{\frac{2(N-1)}{N}}(\partial \Omega)}
  \le h_q(\eps):=
\begin{cases}
c \eps^{-\frac{N}{p+1} + \frac{N}{2}} & \text{ if } q >\frac{N+4}{2(N-2)}\\
c \eps^{-\frac{N}{p+1} + \frac{N}{2}} |\log \eps|^{\frac{N}{2(N-1)}} & \text{ if } q =\frac{N+4}{2(N-2)}\\
c \eps^{-\frac{N}{p+1} +q(N-2)-2} & \text{ if } q < \frac{N+4}{2(N-2)}.
\end{cases} 
\end{equation}
Let \eqref{eq:assumptions_pq3} hold. Then
\begin{equation}\label{stima norma N=4} \norm{\partial_\nu \ue}_{L^{\frac{3}{2}}(\partial \Omega)}
  \le  c \eps^{2\frac{p-1}{p+1}} (\log \eps)^{\frac 23 }. \end{equation}
\end{lemma}
\begin{proof}
Let us first consider the case $N \ge 5$, namely we first prove \eqref{stima norma}. 
 Notice that 
\begin{align}
\partial_\nu U_\eps(x)& = \partial_\nu(\eps^{-\frac{N}{p+1}}U(x/\eps)) =\eps^{-\frac{N}{p+1} - 1} \nabla U(x/\eps) \cdot \nu(x)  = \eps^{-\frac{N}{p+1} - 1}  U'(|x|/\eps) \frac{x}{|x|} \cdot \nu(x)\label{eq:normal_der}.
\end{align}
Now let $\eta>0$ be as in \eqref{eq:boundary close to 0}.  For each $x \in B_\eta(0)\cap \partial \Omega$, one has \eqref{eq:unitarynormal}, thus 
\[ \frac{x}{|x|} \cdot \nu(x) = \frac{\sum_{j=1}^{N-1} 2\rho_j x_j^2-x_N + O(|x'|^3)}{|x|\sqrt{1+O(|x'|^2)}}=\frac{ \sum_{j=1}^{N-1} \rho_j x_j^2 + O(|x'|^3)}{|x|\sqrt{1+O(|x'|^2)}} \]
and therefore
\begin{equation}\label{eq:point_partialnu} 
\abs{ \partial_\nu U_\eps(x)} \le \eps^{-\frac{N}{p+1} - 1}  \abs{U'\left( \frac{|(x', \rho(x')) |}{\eps}\right) }\abs{\frac{ \sum_{j=1}^{N-1} \rho_j x_j^2 + O(|x'|^3)}{|x|}}. 
\end{equation}
We split the norm as follows 
\begin{equation}\label{split} \int_{\partial \Omega} |\partial_\nu U_\eps|^{\frac{2(N-1)}{N}} = \underbrace{\int_{\partial \Omega \cap B_\eta(0)} |\partial_\nu U_\eps|^{\frac{2(N-1)}{N}}}_{(I)} + \underbrace{\int_{\partial \Omega \setminus B_\eta(0)} |\partial_\nu U_\eps|^{\frac{2(N-1)}{N}}}_{(II)}. \end{equation}
We also observe that estimates \eqref{eq:decayestimate} and \eqref{eq:decayestimates2} tell us that
\begin{equation}\label{decay derivative} rU'(r) \approx 
\begin{cases}
r^{2-N} & \text{ if } q >  \frac{N}{N-2}\\
r^{2-N} \log r & \text{ if } q =  \frac{N}{N-2}\\
r^{2-q(N-2)} & \text{ if } q < \frac{N}{N-2}. 
\end{cases} \end{equation}
as $r\to \infty$, and that
\[
\frac{N+4}{2(N-2)}<\frac{N}{N-2}.
\]

\smallbreak 

Step 1. Estimate of $(I)$. The first integral is, calling $\Delta_\eta$ the projection of  
\[ \partial \Omega \cap B_\eta(0)=\{(x',x_N)\in B_\eta(0):\ x_N=\rho(x')=\sum_{j=1}^{N-1} \rho_j x_j^2+O(|x'|^3))\} \] onto the first $N-1$ components, and observing that the $(N-1)$--surface element is $\sqrt{1+|\nabla \rho(x')|^2}=1+O(|x'|)$ and using \eqref{eq:point_partialnu} 
\begin{align}
\nonumber &\int_{\partial \Omega \cap B_\eta} |\partial_\nu U_\eps|^{\frac{2(N-1)}{N}}  \le \eps^{-2 (\frac{N}{p+1} + 1)\frac{N-1}{N}} \int_{\partial \Omega \cap B_\eta(0)} \abs{ U'\left(\frac{|(x', \rho(x'))|}{\eps}\right) \frac{ \sum_{j=1}^{N-1} \rho_j x_j^2 + O(|x'|^3)}{|x'|\sqrt{1+O(|x'|^2)}} }^{\frac{2(N-1)}{N}} \\							&\le  c' \eps^{-2 (\frac{N}{p+1} + 1)\frac{N-1}{N}}  \left( \int_{\Delta_\eta} \abs{ U'\left(\frac{|(x', \rho(x'))|}{\eps}\right) |x'|}^{\frac{2(N-1)}{N}} + \int_{\Delta_\eta} \abs{ U'\left(\frac{|(x', \rho(x'))|}{\eps}\right)\abs{x'}^2}^{\frac{2(N-1)}{N}}  \right) \nonumber \\
&\leq c''\eps^{-2 (\frac{N}{p+1} + 1)\frac{N-1}{N}}  \int_{\Delta_\eta} \abs{ U'\left(\frac{|(x', \rho(x'))|}{\eps}\right) |x'|}^{\frac{2(N-1)}{N}}  \label{stima 1} ,
\end{align}
by choosing $\eta>0$ sufficiently small. 
 
 We now claim that, as $\eps \to 0^+$,
\begin{equation}\label{eq:estimate_of_(I.1)}
    \int_{\Delta_\eta} \abs{ U'\left(\frac{|(x', \rho(x'))|}{\eps}\right) |x'|}^{\frac{2(N-1)}{N}}=\begin{cases}
      O(\eps^{\frac{(N-1)(N+2)}{N}}) & \text{ if }  q > \frac{N+4}{2(N-2)}\\
     O(\eps^{\frac{(N-1)(N+2)}{N}}\log \eps) &\text{ if } q = \frac{N+4}{2(N-2)} \\
     O(\eps^{\frac{2(N-1)}{N}(q(N-2)-1)})  & \text{ if }  q < \frac{N+4}{2(N-2)}
    \end{cases}
\end{equation}
To check this, we change variable $\eps y'=x'$ and use the expression of $\rho$ to obtain
\[ 
c |y'| \ge |(y', \frac{\rho(\eps y')}{\eps})| \ge |y'|. 
\] Then 
\begin{align}
\nonumber \int_{\Delta_\eta} \abs{ U'\left(\frac{|(x', \rho(x'))|}{\eps}\right) |x'|}^{\frac{2(N-1)}{N}} \, dx' &= \eps^{N-1+ \frac{2(N-1)}{N}}\int_{\Delta_\eta/\eps} \abs{U'(|(y', \frac{\rho(\eps y')}{\eps})|) |y'|}^{\frac{2(N-1)}{N}}\, dy'\\
\nonumber &=O(\eps^{N-1+ \frac{2(N-1)}{N}}) \left(\int_0^1 +\int_1^{\eta/\eps} \right) \abs{U'(r) r}^{\frac{2(N-1)}{N} } r^{N-2}\, dr\\
&= O(\eps^{N-1+ \frac{2(N-1)}{N}}) \left(1+\int_1^{\eta/\eps}  \abs{U'(r) r}^{\frac{2(N-1)}{N} } r^{N-2}\, dr \right).\label{eq:change of var} 
\end{align} 
From \eqref{decay derivative} we have
\[
|U'(r)r|^{\frac{2(N-1)}{N}}  r^{N-2}\approx \begin{cases}
r^{-\frac{(N-2)^2}{N}} & \text{ if } q >  \frac{N}{N-2}\\
r^{N-2  +(2-q(N-2))\frac{2(N-1)}{N}} \log r & \text{ if } q =  \frac{N}{N-2}\\
r^{N-2  +(2-q(N-2))\frac{2(N-1)}{N}}& \text{ if } q < \frac{N}{N-2}.
\end{cases}
\]
 Therefore:
\begin{itemize}
    \item Case $q>\frac{N+4}{2(N-2)}$. In this situation the function is integrable for $r\in [1,\infty)$: indeed,  in case $q >  \frac{N}{N-2}$ this is a consequence of the fact that $N\geq 5$, while in case $q \leq   \frac{N}{N-2}$ is simply because $N-2  +(2-q(N-2))\frac{2(N-1)}{N}<-1 \iff q>\frac{N+4}{2(N-2)}$. Hence,
    \begin{equation*} 
    (I.1)= O(\eps^{N-1+ \frac{2(N-1)}{N}}). 
    \end{equation*} 
    \item Case $q=\frac{N+4}{2(N-2)}$. We have $|U'(r)r|^{\frac{2(N-1)}{N}}  r^{N-2} \approx 1/r$ as $r\to \infty$ and
    \begin{align*}
    (I.1)=O(\eps^{N-1+ \frac{2(N-1)}{N}}) \left(1+\int_1^{\eta/\eps} \frac{1}{r}\, dr \right)=O(\eps^{N-1+ \frac{2(N-1)}{N}}) \left(1+\log \eps \right)=O(\eps^{N-1+ \frac{2(N-1)}{N}} \log \eps).
    \end{align*}
    \item Case  $ q < \frac{N+4}{2(N-2)}$. We have
    \begin{align*}
        (I.1)&=O(\eps^{N-1+\frac{2(N-1)}{N}})\left(1+\eps^{-(N-1)-(2-q(N-2)\frac{2(N-1)}{N})}\right)\\
        &=O(1) \left(\eps^\frac{(N-1)(N+2)}{N}+\eps^{\frac{2(N-1)}{N}(q(N-2)-1)}\right)=O(\eps^{\frac{2(N-1)}{N}(q(N-2)-1)}),
    \end{align*}
    since $\frac{(N-1)(N+2)}{N}>\frac{2(N-1)}{N}(q(N-2)-1) \iff q <\frac{N+4}{2(N-2)}$. 
    \end{itemize}
The three cases combined yield \eqref{eq:estimate_of_(I.1)}

Summing up and going back to \eqref{stima 1}, we have
\begin{equation}\label{eq:estimate_ddnU_Step1}
\int_{\partial \Omega } |\partial_\nu U_\eps|^{\frac{2(N-1)}{N}}=
\underbrace{\int_{\partial \Omega \setminus B_\eta(0)} (\partial_\nu U_\eps)^{\frac{2(N-1)}{N}}}_{(II)}+\begin{cases}
 \displaystyle O(\eps^{\frac{(N-1)(p-1)}{p+1}})   & \text{ if } q>\frac{N+4}{2(N-2)}\\
 O(\eps^{\frac{(N-1)(p-1)}{p+1}}\log \eps)   & \text{ if } q=\frac{N+4}{2(N-2)}\\
O(\eps^{-\frac{2 (N-1)}{p+1} +2\frac{N-1}{N}(q(N-2)-2)})& \text{ if } q<\frac{N+4}{2(N-2)}
\end{cases}.
\end{equation}

\smallbreak

Step 2. Estimate of (II). Recalling \eqref{eq:normal_der} and \eqref{decay derivative} we have, for $x\in \partial \Omega \setminus B_\eta(0)$, 
\[ \abs{\partial_\nu U_\eps} \le  \eps^{-\frac{N}{p+1} - 1}  |U'(|x|/\eps)|=
\begin{cases}
O(\eps^{N-2-\frac{N}{p+1}}) & \text{ if } q>\frac{N}{N-2}\\
O(\eps^{N-2-\frac{N}{p+1}}\log \eps) & \text{ if } q=\frac{N}{N-2}\\
O(\eps^{q(N-2)-2-\frac{N}{p+1}}) & \text{ if } q<\frac{N}{N-2}. 
\end{cases}\]
thus  (because $N>4 \iff \frac{N+4}{2(N-2)}<\frac{N }{N-2}$)
\begin{align} \nonumber \int_{\partial \Omega \setminus B_\eta(0)} (\partial_\nu U_\eps)^{\frac{2(N-1)}{N}}
&=\begin{cases}
O(\eps^{2\frac{N-1}{N}(N-2- \frac{N}{p+1}) } )& \text{ if }q>\frac{N}{N-2}\\
 O(\eps^{2\frac{N-1}{N}(N-2- \frac{N}{p+1}) }(\log \eps)^\frac{2(N-1)}{N} )& \text{ if }q=\frac{N}{N-2}\\
 O(\eps^{-\frac{2 (N-1)}{p+1} +2\frac{N-1}{N}(q(N-2)-2)}) & \text{ if }q<\frac{N}{N-2}\\
\end{cases}\\
&\label{eq:estimate_ddnU_Step2} =\begin{cases}
o(\eps^{\frac{(N-1)(p-1)}{p+1}})& \text{ if }q>\frac{N+4}{2(N-2)}\\
o(\eps^{\frac{(N-1)(p-1)}{p+1}}\log \eps)& \text{ if }q=\frac{N+4}{2(N-2)}\\
 O(\eps^{-\frac{2 (N-1)}{p+1} +2\frac{N-1}{N}(q(N-2)-2)}) & \text{ if }q<\frac{N+4}{2(N-2)}\\
\end{cases}.
\end{align}

\smallbreak

Hence, combining \eqref{eq:estimate_ddnU_Step1} with \eqref{eq:estimate_ddnU_Step2} we obtain \eqref{stima norma}.

\smallbreak 

We now consider the case $N=4$, and we prove \eqref{stima norma N=4}. We again split the norm as in \eqref{split}, and reasoning as in \eqref{stima 1} and \eqref{eq:change of var} we get
\[ \int_{\partial \Omega \cap B_\eta} |\partial_\nu U_\eps|^{\frac{3}{2}} \le c \eps^{- \frac{3}{2} (\frac{4}{p+1} + 1)}  \int_{\Delta_\eta} \abs{ U'\left(\frac{|(x', \rho(x'))|}{\eps}\right) |x'|}^{\frac{3}{2}} =  c \eps^{3\frac{p-1}{p+1}} \left(1+\int_1^{\eta/\eps}  \abs{U'(r) r}^{\frac{2(N-1)}{N} } r^{N-2}\, dr \right). \]
We now observe that 
\[ \abs{U'(r) r}^{\frac{2(N-1)}{N} } r^{N-2} \approx r^{-1}, \] 
thus
\[ \int_{\partial \Omega \cap B_\eta} |\partial_\nu U_\eps|^{\frac{3}{2}} \le c \eps^{3\frac{p-1}{p+1}} |\log \eps|. \] 
 The estimate of (II) is exactly as in Step 2 (recall that \eqref{eq:assumptions_pq3} implies $q>2$), so we get
\[ \int_{\partial \Omega \setminus B_\eta(0)} (\partial_\nu U_\eps)^{\frac{3}{2}} \le c \eps^{3\frac{p-1}{p+1}}=o(\eps^{3\frac{p-1}{p+1}} |\log \eps|). \]
We now immediately deduce \eqref{stima norma N=4}. 
\end{proof}

\begin{lemma}\label{eq:estimate_normal_derivative2}
Assume $p,q$ satisfy \eqref{eq:assumptions_pq} or \eqref{eq:assumptions_pq2}. One has
\begin{equation*}\label{stima norma2}
\norm{\partial_\nu \ve}_{L^{\frac{2(N-1)}{N}}(\partial \Omega)}
  \leq c \eps^{-\frac{N}{q+1} + \frac{N}{2}}
\end{equation*}
If \eqref{eq:assumptions_pq3} holds, then 
\[ \norm{\partial_\nu \ve}_{L^{\frac{3}{2}}(\partial \Omega)}
  \leq c \eps^{2 \frac{q-1}{q+1}}(\log \eps)^{\frac 23 }. \]
\end{lemma}
\begin{proof}
The proof is quite similar to the one of Lemma \ref{eq:estimate_normal_derivative}, to which we refer for the notation and more details. We have
\begin{align*}
\int_{\partial \Omega \cap B_\eta(0)} |\partial_\nu V_\eps|^\frac{2(N-1)}{N}& \leq  c\eps^{-2 (\frac{N}{q+1} + 1)\frac{N-1}{N}}  \int_{\Delta_\eta} \abs{ V'\left(\frac{|(x', \rho(x'))|}{\eps}\right) |x'|}^{\frac{2(N-1)}{N}}  \\
&\leq c' \eps^{-2 (\frac{N}{q+1} + 1)\frac{N-1}{N}} \eps^{N-1+ \frac{2(N-1)}{N}} \left(1+\int_1^{\eta/\eps}  \abs{V'(r) r}^{\frac{2(N-1)}{N} } r^{N-2}\, dr \right) \\
&=c' \eps^{\frac{(N-1)(q-1)}{q+1}}\left(1+\int_1^{\eta/\eps}  \abs{V'(r) r}^{\frac{2(N-1)}{N} } r^{N-2}\, dr \right).
\end{align*}
Now estimates \eqref{eq:decayestimate} and \eqref{eq:decayestimates2} provide
\[
rV'(r)\approx r^{2-N}, \quad \text{ so that } \quad |rV'(r)|^\frac{2(N-1)}{N}r^{N-2}\approx r^{-\frac{(N-2)^2}{N}},
\]
which is integrable at infinity if $N \ge 5$, and is equal to $r^{-1}$ if $N=4$. Then if $N \ge 5$
\[
\int_\Omega |\partial_\nu V_\eps|^\frac{2(N-1)}{N} = \int_{\partial \Omega \cap B_\eta(0)} |\partial_\nu V_\eps|^\frac{2(N-1)}{N}+\int_{\partial \Omega \setminus B_\eta(0)} |\partial_\nu V_\eps|^\frac{2(N-1)}{N}=O(\eps^\frac{(N-1)(q-1)}{(q+1)})
\]
whereas if $N=4$
\[
\int_\Omega |\partial_\nu V_\eps|^\frac{3}{2} =O(\eps^{3\frac{q-1}{q+1}}|\log \eps|)
\]
and the proof is finished.
\end{proof}

\section*{Acknowledgments.}

Delia Schiera and Hugo Tavares are partially supported by the Portuguese government through FCT - Funda\c c\~ao para a Ci\^encia e a Tecnologia, I.P., under the projects UID/MAT/04459/2020 and PTDC/MAT-PUR/1788/2020. Delia Schiera is also partially supported by FCT, I.P. and, when eligible, by COMPETE 2020 FEDER funds, under the Scientific Employment Stimulus - Individual Call (CEEC Individual) - 2020.02540.CEECIND/CP1587/CT0008.

\noindent \textbf{Angela Pistoia}\\
Dipartimento di Scienze di Base e Applicate per l’Ingegneria\\
Sapienza Universita di Roma\\
Via Scarpa 16, 00161 Roma, Italy\\
\texttt{angela.pistoia@uniroma1.it} 
\vspace{.3cm}

\noindent \textbf{Delia Schiera, Hugo Tavares}\\
Departamento de Matemática do Instituto Superior Técnico\\
Universidade de Lisboa\\
Av. Rovisco Pais\\
1049-001 Lisboa, Portugal\\
\texttt{delia.schiera@tecnico.ulisboa.pt, hugo.n.tavares@tecnico.ulisboa.pt}


\begin{thebibliography}{9}
\bibitem{Adams} R. Adams and  J.Fournier, \textit{Sobolev Spaces}. Second edition. Pure and Applied Mathematics (Amsterdam), 140. Elsevier/Academic Press, Amsterdam, 2003. 
\bibitem{Agmon} S. Agmon, A. Douglis and L. Nirenberg, Estimates near the boundary for solutions of elliptic partial differential equations satisfying general boundary conditions. I, \textit{Comm. Pure Appl. Math.} \textbf{12}: 623--727, 1959.
 \bibitem{AY} A.I. Ávila, J. Yang, On the existence and shape of least energy solutions for some elliptic systems, \textit{J. Differential
Equations} \textbf{191}(2): 348--376, 2003.
\bibitem{BCN} D. Bonheure, H. Cheikh Ali and R. Nascimento, A Paneitz-Branson type equation with Neumann boundary conditions, \textit{Adv. Calc. Var.} \textbf{14}(4):499--519, 2021.
\bibitem{BST} D. Bonheure, E. Moreira dos Santos,  and H. Tavares, Hamiltonian elliptic systems: a guide to variational frameworks, \textit{Port. Math.} \textbf{71}, no. 3-4, 301--395, 2014.
\bibitem{BSTilli}  D. Bonheure, E. Serra, P. Tilli, Radial positive solutions of elliptic systems with Neumann boundary conditions,
\textit{J. Funct. Anal.} \textbf{265}(3):375--398, 2013.
\bibitem{ChenJinLiLim} W. Chen, C. Chao, C. Li, J. Lim, Weighted Hardy-Littlewood-Sobolev inequalities and systems of integral equations,  \textit{Discrete Contin. Dyn. Syst.} 2005, suppl., 164--172.
\bibitem{ChenLi} W. Chen, C. Li, Regularity of solutions for a system of integral equations. \textit{Commun. Pure Appl. Anal.} \textbf{4}(1):1--8, 2005.
\bibitem{Cherrier} P. Cherrier, Meilleures constantes dans des inégalités relatives aux espaces de Sobolev, \textit{Bull. Sci. Math.} (2), \textbf{108}(3):225--262, 1984.
\bibitem{ChoiKim} W. Choi, S. Kim, Asymptotic behavior of least energy solutions to the Lane-Emden system near the critical hyperbola. \textit{J. Math. Pures Appl.}(9) \textbf{132}:398--456, 2019. 
\bibitem{ClappSaldana} M. Clapp, A. Salda\~na, Entire nodal solutions to the critical Lane-Emden system. \textit{Comm. Partial Differential Equations} \textbf{45}, no. 4, 285--302, 2020. 
\bibitem{CFM} P. Clément, D.G. de Figueiredo, E. Mitidieri, E. 
Positive solutions of semilinear elliptic systems. 
\textit{Comm. Partial Differential Equations} \textbf{17}, no. 5-6, 923--940, 1992. 
\bibitem{CK} M. Comte and M. C. Knaap, Existence of solutions of elliptic equations involving critical Sobolev exponents with Neumann boundary condition in general domains, \textit{Differential Integral Equations}, \textbf{4}(6):1133--1146, 1991.
\bibitem{DRW} O. Druet, F. Robert, J. Wei, The Lin-Ni's problem for mean convex domains, Mem. Amer. Math. Soc. \textbf{218}, no. 1027, vi+105 pp., 2012. 
\bibitem{FKP} R.L. Frank, S. Kim, A. Pistoia, 
Non-degeneracy for the critical Lane-Emden system.
\textit{Proc. Amer. Math. Soc.} \textbf{149}, no. 1, 265--278, 2021. 
\bibitem{GT} D. Gilbarg and N.S. Trudinger. Elliptic partial differential equations of second order. Reprint of the 1998 edition. Classics in Mathematics. Springer-Verlag, Berlin, 2001. xiv+517 pp.
\bibitem{Guerra} I.A. Guerra,
Solutions of an elliptic system with a nearly critical exponent. 
\textit{Ann. Inst. H. Poincaré C Anal. Non Linéaire} \textbf{25}, no. 1, 181--200, 2008
\bibitem{HMV} J. Hulshof, E. Mitidieri and  R. Vandervorst, Strongly indefinite systems with critical Sobolev exponents, \textit{Trans. Amer. Math. Soc.} \textbf{350}(6):2349--2365, 1998.
\bibitem{HV}  J. Hulshof and R. Vandervorst, Asymptotic behaviour of ground states, \textit{Proc. Amer. Math. Soc.} \textbf{124}:2423--2431, 1996.
\bibitem{KimCoron} S. Jin and S. Kim,  Coron's problem for the critical Lane-Emden system, preprint arXiv:2210.13068. 
\bibitem{KM} S. Kim, S.-H. Moon, Asymptotic analysis on positive solutions of the Lane-Emden system with nearly critical exponents, \textit{Trans. Amer. Math. Soc.}, 2023.
\bibitem{Pistoia} S. Kim and A. Pistoia, A perturbative approach to non-degeneracy of the Lane-Emden system, preprint arXiv:1907.11303. 
\bibitem{KP} S. Kim, A. Pistoia,
Multiple blowing-up solutions to critical elliptic systems in bounded domains,
\textit{J. Funct. Anal.} \textbf{281}(2), Paper No. 109023, 58 pp, 2021.
\bibitem{LFMS} J. Lange Ferreira Melo, E. Moreira dos Santos, Critical and noncritical regions on the critical hyperbola. Contributions to nonlinear elliptic equations and systems, pp.345–-370, Progr. Nonlinear Differential Equations Appl., 86, Birkhäuser/Springer, Cham, 2015. 
\bibitem{LiebLoss} E. Lieb and M. Loss, \text{Analysis}. Second edition. Graduate Studies in Mathematics, 14. American Mathematical Society, Providence, RI, 2001. 
\bibitem{Lions}  P.-L. Lions, The concentration-compactness principle in the calculus of variations. The limit case. I, \textit{Rev. Mat. Iberoamericana} \textbf{1}(1):145--201, 1985. 
\bibitem{Mitidieri} E. Mitidieri, A Rellich type identity and applications. \textit{Comm. Partial Differential Equations} \textbf{18}, no. 1-2, 125--151, 1993. 
\bibitem{PW} E. Parini, T. Weth, Existence, unique continuation and symmetry of least energy nodal solutions to sublinear Neumann problems. \textit{Math. Z.} \textbf{280}, no. 3-4, 707--732, 2015.
\bibitem{PvV}  L.A. Peletier, R.C.A.M. Van der Vorst,  Existence and nonexistence of positive solutions of nonlinear elliptic systems and the biharmonic equation. \textit{Differential Integral Equations} \textbf{5}, no. 4, 747--767, 1992. 
\bibitem{PR}  A. Pistoia, M. Ramos, Locating the peaks of the least energy solutions to an elliptic system with Neumann boundary
conditions, \textit{J. Differential Equations} \textbf{201} (1) 160–-176, 2004.
\bibitem{PST} A. Pistoia, A. Salda\~na and H. Tavares, Existence of solutions to a slightly supercritical pure Neumann problem,  preprint arXiv:2209.02113.
\bibitem{RY}  M. Ramos, J. Yang, Spike-layered solutions for an elliptic system with Neumann boundary conditions, \textit{Trans. Amer.
Math. Soc.} \textbf{357}(8):3265--3284, 2005. 
\bibitem{Rassias} G.~M. Rassias and  T.~M. Rassias (Eds.), Differential Geometry, Calculus of Variations, and Their Applications, Lect. Notes Pure Appl. Math., vol. 100, Marcel Dekker, Inc., New York, 1985.
\bibitem{reywei} O. Rey, J. Wei, Arbitrary number of positive solutions for an elliptic problem with critical nonlinearity, \textit{J. Eur. Math. Soc. (JEMS)} \textbf{7}:449--476, 2005.
\bibitem{ST2} A. Salda{\~n}a and H. Tavares, Least energy nodal solutions of Hamiltonian elliptic systems with Neumann boundary conditions, \textit{J. Differential Equations}, \textbf{265}(12):6127--6165, 2018.
\bibitem{ST} A. Salda{\~n}a and H. Tavares, On the least-energy solutions of the pure Neumann Lane--Emden equation, \textit{NoDEA - Nonlinear Differential Equations and Applications} \textbf{29}, Article number: 30 (2022). 
\bibitem{SW} A. Salda{\~n}a, T. Weth, Asymptotic axial symmetry of solutions of parabolic equations in bounded radial domains, \textit{J. Evol. Equ.} \textbf{12}, no. 3, 697--712, 2012.
\bibitem{VanderVorst} R.C.A.M. Van der Vorst, Variational identities and applications to differential systems. \textit{Arch. Rational Mech. Anal.} \textbf{116}, no. 4, 375--398, 1992. 
\bibitem{Willem} M. Willem, Minimax theorems. Progress in Nonlinear Differential Equations and their Applications, 24. Birkhäuser Boston, Inc., Boston, MA, 1996. 
\bibitem{Zeng}  J. Zeng, The estimates on the energy functional of an elliptic system with Neumann boundary conditions, \textit{Bound.
Value Probl.} (194):11, 2013. 
\end{thebibliography}
\end{document}